\documentclass{amsart}
\usepackage{fancyvrb}
\usepackage{amsmath}
\usepackage{amsfonts}
\usepackage{amsthm}
\usepackage{cases}
\usepackage{caption}
\usepackage{xcolor}
\usepackage{tikz}
\usepackage{centernot}
\usepackage{mathtools}
\usepackage[margin=1in]{geometry}
\usepackage[makeroom]{cancel}

\newtheorem{lemma}{Lemma}[section]

\newtheorem{conjecture}[lemma]{Conjecture} 
\newtheorem{proposition}[lemma]{Proposition} 
\newtheorem{theorem}[lemma]{Theorem}

\theoremstyle{definition}
\newtheorem{definition}{Definition}[section]
\newtheorem{example}{Example}[section]

\theoremstyle{remark}

\numberwithin{equation}{section}

\title[A Molev-Sagan type formula]{A Molev-Sagan type formula for double Schubert polynomials}
\author{Matthew J. Samuel}
\address{Monroe Township, Middlesex County, New Jersey, United States}

\email{matthematics@gmail.com}

\DeclareMathOperator{\sch}{\mathfrak{S}}
\newcommand{\tom}[1]{\xrightarrow{#1}}

\begin{document}

\begin{abstract}
We give a Molev-Sagan type formula for computing the product $\sch_u(x;y)\sch_v(x;z)$ of two double Schubert polynomials in different sets of coefficient variables where the descents of $u$ and $v$ satisfy certain conditions that encompass Molev and Sagan's original case and conjecture positivity in the general case. Additionally, we provide a Pieri formula for multiplying an arbitrary double Schubert polynomial $\sch_u(x;y)$ by a factorial elementary symmetric polynomial $E_{p,k}(x;z)$. Both formulas remain positive in terms of the negative roots when we set $y=z$, so in particular this gives a new equivariant Littlewood-Richardson rule for the Grassmannian, and more generally a positive formula for multiplying a factorial Schur polynomial $s_{\lambda}(x_1,\ldots,x_m;y)$ by a double Schubert polynomial $\sch_v(x_1,\ldots,x_p;y)$ such that $m\geq p$. An additional new result we present is a combinatorial proof of a conjecture of Kirillov of nonnegativity of the coefficients of skew Schubert polynomials, and we conjecture a weight-preserving bijection between a modification of certain diagrams used in our formulas and RC-graphs/pipe dreams arising in formulas for double Schubert polynomials.

\end{abstract}
\maketitle

\section{Introduction}

Schubert polynomials were originally defined by Lascoux and Sch\"utzenberger in \cite{lsschub}, and double Schubert polynomials are a generalization found in \cite{notes}. The double Schubert polynomials are polynomials in two infinite sets of variables $\{x_i:i\in\mathbb{N}\}$ and $\{y_i:i\in\mathbb{N}\}$ with integer coefficients and are denoted by $\sch_u(x;y)$ for permutations $u$. Double Schubert polynomials are linearly independent and in fact form a basis of the polynomial ring over the coefficient ring $\mathbb{Z}[y]$. Double Schubert polynomials are interesting first and foremost because they represent Schubert classes in the torus-equivariant cohomology ring of the complete flag variety. If we introduce a third infinite set of variables $\{z_i:i\in \mathbb{N}\}$, we may write the product of two double Schubert polynomials with different sets of coefficient variables, defining coefficients $c_{uv}^w(y;z)$ as follows:
$$\sch_u(x;y)\sch_v(x;z)=\sum_{w\in S_\infty}{c_{uv}^w(y;z)\sch_w(x;y)}$$

We have the following conjecture:
\begin{conjecture} \label{conjecture:positive}
For all $u,v,w$ we have that $c_{uv}^w(y;z)$ is a polynomial in the differences $y_i-z_j$ with nonnegative integer coefficients. 
\end{conjecture}
At the time of this writing, we have computationally verified this conjecture for all $u,v\in S_5$ for all $w\in S_\infty$. The truth of this conjecture was demonstrated for factorial Schur functions in \cite{molev1999littlewood} via a Littlewood-Richardson rule. $c_{uv}^w(0;0)$ is known to be nonnegative as these are the structure constants in the ordinary cohomology of the complete flag variety, which count points in triple intersections of Schubert varieties. No combinatorial proof is known of this in general, but several widely applicable combinatorial results are known; the full proof is obtained using algebraic geometry. It is also known \cite{graham2001positivity} via another algebraic geometry proof that $c_{uv}^w(y;y)$ is a polynomial in the differences $y_{i+1}-y_i$ with nonnegative integer coefficients, which was shown combinatorially for the Grassmannian in \cite{knutson2003puzzles} and for the two-step flag variety in \cite{buch2015mutations}. It was conjectured by Kirillov \cite{kirillov2007skew} (essentially) that $c_{uv}^w(y;0)$ is a polynomial in $y$ with nonnegative integer coefficients for all $u,v,w$ and proved in the same article that the conjecture holds when $\ell(u,w)=1$. This is clearly a special case of Conjecture \ref{conjecture:positive}. We have verified Kirillov's conjecture, which is easier to test, for $u,v\in S_7$ for all $w\in S_\infty$. Kirillov's conjecture was presented in an alternative form by the author on MathOverflow \cite{question}, and in a comment Dave Anderson suggested that we likely have the tools to prove this with the current state of knowledge, though as far as we know this has not yet been done. We present a combinatorial proof of an additional special case (conjectured separately in Kirillov's article) in Theorem \ref{theorem:kirpositive}.

This article has two main results. The first, Theorem \ref{theorem:main}, is a positive combinatorial formula for $c_{uv}^w(y;z)$ where there exists a positive integer $p$ such that $\ell(us_i)>\ell(u)$ whenever $i<p$ and $\ell(vs_i)>\ell(v)$ whenever $i>p$. Equivalently, it is a positive formula as a polynomial in the differences $y_i-z_j$ for coefficients of $\sch_w(x;y)$ in the expansion of the product 
$$\sch_u(x_1,\ldots,x_m;y)\sch_v(x_1,\ldots,x_p;z)$$ 
where $\sch_u(x_1,\ldots,x_m;y)$ is symmetric in the variables $x_1,\ldots,x_p$. The formula when substituting to obtain $c_{uv}^w(y;y)$ is also nonnegative (though some terms are $0$) in the sense of \cite{graham2001positivity}, which we show in Theorem \ref{theorem:equivpositive}, thus also giving a positive formula for ordinary double Schubert polynomial multiplication for these pairs of permutations; all of the nonzero terms in the formula for $c_{uv}^w(y;y)$ are visibly polynomials in $y_{i+1}-y_i$ with nonnegative integer coefficients. 

The second main result, Theorem \ref{theorem:pieri}, is a Pieri formula for computing $c_{uv}^w(y;z)$ when $v$ is a cycle of the form $c_{p,k}=s_{k-p+1}s_{k-p+2}\cdots s_k$, in which case $\sch_{c_{p,k}}(x;z)$ is a factorial elementary symmetric polynomial, for which we introduce the alternative notation $E_{p,k}(x;z)$. This generalizes the main result of \cite{sottile} and is closely related to the equivariant Pieri formula in \cite{robinsonpieri}. Via the formula in \cite[(4)]{molev1999littlewood} applied to a column shape, which we state here as an explicit formula for $E_{p,k}(x;z)$ (Proposition \ref{proposition:elemformula}), our Pieri formula is positive in the sense of Conjecture \ref{conjecture:positive} and also positive for $c_{uv}^w(y;y)$ in the sense of \cite{graham2001positivity}.

In Section \ref{section:moregeneral} we state a more general result than Theorem \ref{theorem:main} (Theorem \ref{theorem:moregeneral}) that differs only in the choices of the permutations $u$ and $v$; the formula is the same as that of the main result. Theorem \ref{theorem:moregeneral} is the most general case to which our method applies.

\section{Preliminaries}
\newcommand{\code}{\mathfrak{c}}
\newcommand{\monom}[1]{\mu_{#1}}
\begin{definition}[The symmetric group]
There are many references regarding the algebraic combinatorics of symmetric groups as Coxeter groups, for which we will cite \cite{combcox}. We use the notation $S_\infty$ for the infinite symmetric group of bijections $\mathbb{N}\to\mathbb{N}$ fixing all but finitely many elements, with function composition in the usual order being the operation, and $S_n<S_\infty$ is the subgroup of all $u\in S_\infty$ such that $u(i)=i$ if $i>n$. We think of elements of $S_\infty$ as sequences and refer to them as \emph{permutations}. Every element $u\in S_\infty$ is in $S_n$ for some $n$, and if $u\in S_n$ we may write
$$u=[u(1),u(2),\ldots,u(n)]$$
(this is known as ``window notation''). Clearly if $u\in S_n$ then $u\in S_N$ for all $N>n$ as well.

$s_i\in S_\infty$ is the adjacent transposition (also known as a simple reflection) exchanging the values $i$ and $i+1$, and $t_{ab}$ for positive integers $a\neq b$ is the (not necessarily adjacent) transposition exchanging $a$ and $b$ and fixing everything else. Note that we do not impose an order on $a$ and $b$ in the definition of a transposition, so $t_{ab}=t_{ba}$. Every element of $S_\infty$ can be written as a product of adjacent transpositions. For $u\in S_\infty$, $us_i$ is $u$ with the indices $i$ and $i+1$ flipped, and $s_iu$ exchanges the values instead. The \emph{length} $\ell(u)$ is the number of inversions of $u$, the pairs of indices $(i,j)$ with $i<j$ such that $u(i)>u(j)$. For elements $u,w$ with $\ell(u)\leq \ell(w)$ we define $\ell(u,w)=\ell(w)-\ell(u)$. A \emph{word} for an element $u\in S_\infty$ is a sequence of simple reflections 
$$(s_{i_1},\ldots,s_{i_m})$$
such that
$$s_{i_1}\cdots s_{i_m}=u$$
and the word is said to be \emph{reduced} if it is as short as possible. It is well known that a reduced word for $u$ is of length exactly $\ell(u)$. If $u(i)>u(i+1)$, then $\ell(us_i)=\ell(u)-1$, and if $u(i)<u(i+1)$ then $\ell(us_i)=\ell(u)+1$. Indices $i$ with $u(i)>u(i+1)$ are called \emph{(right) descents} of $u$.

The \emph{(right) weak Bruhat order} $\leq$ on permutations is the partial ordering that is the reflexive, transitive closure of the relation $v\trianglelefteq vs_i$ whenever $\ell(vs_i)=\ell(v)+1$. In general, if $v\leq w$, then $\ell(v^{-1}w)=\ell(w)-\ell(v)$.

We denote by $w_0(n)$ the longest element (element with the most inversions) of $S_{n+1}$. This is easily seen to be
$$w_0(n)(i)=\begin{cases}
n+2-i&\mbox{ if }i\leq n+1\\
i&\mbox{ if }i>n+1
\end{cases}$$
\end{definition}

\begin{definition}[Double Schubert polynomials, $\partial^w$, $\partial_u^w$]
A simple definition of double Schubert polynomials is through divided difference operators as defined in \cite{bgg}. $S_\infty$ acts on the polynomial ring by extending via automorphisms the rule
$$u(x_j)=x_{u(j)}$$
We also insist that $u(y_j)=y_j$ and $u(z_j)=z_j$ for all $j$. We define the divided difference operator $\partial^{s_i}$ by
$$\partial^{s_i}(P) = \frac{P-s_i(P)}{x_i-x_{i+1}}$$
It is well known that applying divided difference operators to a polynomial yields a polynomial; this is proved, for example, in \cite{samuelleibniz}. 

The divided difference operators can be composed to yield operators indexed by the symmetric group, $\partial^u$ for $u\in S_\infty$, defined by
$$\partial^u=\partial^{s_{i_1}}\cdots\partial^{s_{i_{\ell(u)}}}$$
where $(s_{i_1},\ldots,s_{i_{\ell(u)}})$ is a reduced word for $u$. This does not depend on the choice of reduced word. If $\ell(us_i)>\ell(u)$, then
$$\partial^u\partial^{s_i}=\partial^{us_i}$$
If $\ell(us_i)<\ell(u)$, then
$$\partial^u\partial^{s_i}=0$$
Divided difference operators satisfy the Leibniz formula
\begin{equation} \label{equation:leibniz}
\partial^{s_i}(PQ)=\partial^{s_i}(P)s_i(Q)+P\partial^{s_i}(Q)
\end{equation}
which can be seen by direct computation. There exist skew divided difference operators $\partial_u^w$ (the same as those defined in \cite{samuelleibniz}, differing by a permutation from those in \cite{notes} and \cite{kirillov2007skew}) defined by the more general Leibniz formula
$$\partial^w(PQ)=\sum_u\partial^u(P)\partial_u^w(Q)$$
We identify divided difference operators, group elements, and skew divided difference operators with the corresponding elements of the nil-Hecke ring \cite{kostant1986nil}, which is a free left module over $\mathbb{Z}[x]$ with basis $\{\partial^u\mid u\in S_\infty\}$. For example,
$$s_i=1+(x_{i+1}-x_{i})\partial^{s_i}$$

It is shown in \cite{samuelleibniz} that
$$\partial_u^w=\sum_v c_{uv}^w(x;x)\partial^v$$
More usefully for this article, skew divided difference operators satisfy the following recurrence relation. For the base case, 
$$\partial_1^1=1$$
and if $u\in S_\infty$ satisfies $u\neq 1$ then
$$\partial_u^1=0$$
Suppose $\ell(ws_i)<\ell(w)$. If $\ell(us_i)<\ell(u)$, then
$$\partial_u^w=\partial_u^{ws_i}\partial^{s_i}+\partial_{us_i}^{ws_i}s_i$$
and if $\ell(us_i)>\ell(u)$ then
$$\partial_u^w=\partial_u^{ws_i}\partial^{s_i}$$

We define, for any $n>0$,
$$\sch_{w_0(n)}(x;y)=\prod_{i+j\leq n+1}(x_i-y_j)$$
Then the double Schubert polynomial $\sch_u(x;y)$ for $u\in S_{n+1}$ is given by
$$\sch_u(x;y)=\partial^{u^{-1}w_0(n)}(\sch_{w_0(n)}(x;y))$$
This definition is shown to yield a well-defined polynomial (i.e., not depending on the choice of $n$) by Macdonald in \cite{notes}. 
\end{definition}

Note that if $\ell(us_i)>\ell(u)$, then $\sch_u(x;y)$ is symmetric in the variables $x_i$ and $x_{i+1}$ and
$$\partial^{s_i}(\sch_u(x;y))=0$$
If $\ell(us_i)<\ell(u)$, then
$$\partial^{s_i}(\sch_u(x;y))=\sch_{us_i}(x;y)$$
We also have the vanishing formula \cite[(6.4)]{notes}
$$\sch_u(y;y)=\delta_{1,u}$$
The collection of $\sch_u(x;y)$ for $u\in S_\infty$ forms a basis for $\mathbb{Z}[x,y,z]$ as a module over $\mathbb{Z}[y,z]$, and by the vanishing formula, applying the divided difference $\partial^w$ then setting $x=y$ allows us to pull out the coefficient of $\sch_w(x;y)$ in the expansion of any polynomial in $\mathbb{Z}[x,y,z]$. 

The skew divided difference operators allow us to obtain coefficients of products of arbitrary polynomials with double Schubert polynomials in $x$ and $y$ as follows. Using the Leibniz formula, for an arbitrary polynomial $P$ and a permutation $u$ we obtain
$$\partial^w(P\sch_u(x;y))=\sum_v\partial_v^w(P)\partial^v(\sch_u(x;y))$$
Again applying the vanishing formula, we obtain the result that substituting $x=y$ in the polynomial $\partial_u^w(P)$ extracts the coefficient of $\sch_w(x;y)$ in the expansion of $P\sch_u(x;y)$. In particular, setting $P=\sch_v(x;z)$, we obtain
$$c_{uv}^w(x;z)=\partial_u^w(\sch_v(x;z))$$
for all $u,v,w\in S_\infty$.

\begin{definition}[Factorial Schur polynomials, Grassmannian permutations]
A permutation $u\in S_\infty$ is said to be \emph{Grassmannian} if it has at most one descent. A factorial Schur polynomial is a double Schubert polynomial corresponding to a Grassmannian permutation \cite[Theorem~4]{bump2011factorial}. Factorial Schur polynomials are indexed by partitions and denoted by $s_\lambda(x_1,\ldots,x_m;y)$. The partition $\lambda$ together with the value of $m$ determine the permutation for the double Schubert polynomial to which $s_\lambda(x_1,\ldots,x_m;y)$ corresponds. Specifically, we define a Grassmannian permutation $w_{\lambda;m}$ corresponding to the partition $\lambda$ with descent at position $m$ as follows. Let
$$\lambda=(\lambda_1,\ldots,\lambda_p)$$
and suppose $m\geq p$. Define $\lambda_i=0$ if $i>p$, then let $w_{\lambda;m}(i)=i+\lambda_{m+1-i}$ if $1\leq i\leq m$. The remainder of $w_{\lambda;m}$ is the complement of $w_{\lambda;m}([m])$ arranged in increasing order. For example, 
$$w_{(3,1,1);4}=[1,3,4,7,2,5,6]$$
Given this definition of $w_{\lambda;m}$, we have the formula
$$s_\lambda(x_1,\ldots,x_m;y)=\sch_{w_{\lambda;m}}(x;y)$$
The \emph{factorial elementary symmetric polynomial} $E_{p,k}(x;y)$ is the special case $s_{(1^p)}(x_1,\ldots,x_k;y)$, about which we go into much more detail below.

Molev and Sagan's Littlewood-Richardson rule \cite{molev1999littlewood} computes the product of two factorial Schur polynomials $s_\lambda(x_1,\ldots,x_m;y)s_\mu(x_1,\ldots,x_m;z)$ where the corresponding permutations have the same descent, i.e. the polynomials have the same number of $x$ variables.
\end{definition}

\begin{definition}[The code $\code(v)$, dominant permutations, the dominant approximation $\monom{v}$, and the partition $\lambda(v)$] \label{definition:permstuff}
We define the \emph{code} $\code(v)$ of a permutation $v\in S_\infty$, a sequence indexed by $i\geq 1$, as follows:
$$\code_i(v)=\#\{j>i\mid v(i)>v(j)\}$$
Note that if $v\in S_{n+1}$, then $\code_i(v)\leq n+1-i$, and it is not hard to see that
$$\ell(v)=\sum_{i=1}^\infty \code_i(v)$$
Also, $\ell(vs_i)=\ell(v)+1$ if and only if $\code_i(v)\leq \code_{i+1}(v)$, in which case $\code_i(vs_i)=\code_{i+1}(v)+1$, $\code_{i+1}(vs_i)=\code_i(v)$, and $\code_j(vs_i)=\code_j(v)$ for $j\notin \{i,i+1\}$.

A permutation $v$ is said to be \emph{dominant} if $\code(v)$ is a partition. For $v\in S_\infty$, we define a dominant permutation $\monom{v}$, the \emph{dominant approximation} of $v$, as follows. If $v$ is dominant, let $\monom{v}=v$. If $v$ is not dominant, let $i$ be the maximal index such that $\code_i(v)<\code_{i+1}(v)$, and define $\monom{v} = \monom{vs_i}$. It is clear that this recursion terminates in a unique dominant permutation. Note that at each step $\ell(vs_i)=\ell(v)+1$, so that $v\leq \monom{v}$ and hence $\ell(v^{-1}\monom{v})=\ell(\monom{v})-\ell(v)$. Define a partition $\lambda(v)$ to be the conjugate of the code of $\monom{v}$, meaning
$$\lambda_i(v)=\#\{j\mid \code_j(\monom{v})\geq i\}$$
\end{definition}

\begin{example}
Let $v=[1,3,5,2,4]$. Then
$$\code(v)=(0,1,2)$$
In the recursion computing the dominant approximation, the first reflection to apply is at index $2$. We have that
$$vs_2 = [1,5,3,2,4]$$
and
$$\code(vs_2)=(0,3,1)$$
Continuing in this manner, we obtain that the remaining reflections are $s_1,s_2$, and hence $\mu_v = vs_2s_1s_2$. We have
$$\mu_v=[5,3,1,2,4]$$
$$\code(\mu_v)=(4,2)$$
which is a partition. $\lambda(v)$ is the conjugate partition, namely
$$\lambda(v)=(2,2,1,1)$$
\end{example}

\begin{definition}[$\tom{k}$, $P_k(u,w)$, $\mathrm{Path}_\lambda(u,w)$, $\mathrm{weight}_{P,\lambda}(y;z)$] \label{definition:pieri}
We define a relation $\tom{k}$ on pairs of permutations, introduced by Sottile \cite{sottile}, which was used originally to indicate when the coefficient of a Schubert polynomial is nonzero in a suitable application of the Pieri formula. In the generalization to double Schubert polynomials in \cite{robinsonpieri}, the same relation applies to all nonzero coefficients in the expansion of the generalized Pieri formula. The relation plays the same role in our case.

Given $u,u'\in S_\infty$ and a positive integer $k$, we declare that $u\tom{k} u'$ if there exists a $p$ with $0\leq p\leq k$ such that there are transpositions $t_{a_1b_1},\ldots,t_{a_pb_p}$ satisfying 
\begin{enumerate}
\item $a_i\leq k<b_i$ for all $i$
\item $a_i\neq a_j$ whenever $i\neq j$
\item $\ell(ut_{a_1b_1}\cdots t_{a_ib_i})=\ell(u)+i$ for all $i$
\item $u'=ut_{a_1b_1}\cdots t_{a_pb_p}$
\end{enumerate}
If $u\tom{k} w$, define 
$$P_k(u,w)=\{u(i)\mid i\leq k\mbox{ and }u(i)=w(i)\}$$
For an integer $j$ we define the weight of this interval by
$$\mathrm{weight}_{u,k}^w(y;z_j)=\prod_{i\in P_k(u,w)}(y_i-z_j)$$
where by convention the empty product is $1$.

Given $u,w\in S_\infty$ and a partition $\lambda$ of length $m$, define $\mathrm{Path}_\lambda(u,w)$ to be the set of sequences defined by
$$\mathrm{Path}_\lambda(u,w)=\{(u_0,u_1,\ldots,u_m)\mid u=u_0\tom{\lambda_1}u_1\tom{\lambda_2}\cdots\tom{\lambda_m}u_m = w\}$$
Then given such a path $P$ and the partition $\lambda$ we define
$$\mathrm{weight}_{P,\lambda}(y;z)=\prod_{i=1}^m\mathrm{weight}_{u_{i-1},\lambda_i}^{u_i}(y;z_i)$$

For illustration in examples, we will represent the elements of $\mathrm{Path}_\lambda(u,w)$ with certain diagrams of numbers. In the diagrams, the permutations $(u_0,\ldots,u_m)$ in the paths will be written vertically in each column, with the columns progressing according to the permutations in the path from left to right. More specifically, the element $u_j(i)$ at the $i$th index in the permutation $u_j$ will be in row $i$ and column $j$, with the column numbers starting at $0$ and row numbers starting at $1$. In column $j$, a horizontal line will be drawn below row $\lambda_j$. The entries that contribute factors to the weight are those where $u_{j-1}(i)=u_j(i)$ for rows $i$ above the horizontal line in column $j$, meaning the number is the same as the number in the same row in the column immediately to the left. These entries will be circled in the diagram. The circled entries in column $j$ represent elements of $P_{\lambda_j}(u_{j-1},u_j)$. If the circled number is $i$ (the value, not the row number) in column $j$, a term of $y_{i}-z_j$ will occur as a factor in the product computing the weight of the path. The weight $\mathrm{weight}_{P,\lambda}(y;z)$ will be written beneath the diagram.

\end{definition}

\newcommand*\circled[1]{\tikz[baseline=(char.base)]{
            \node[shape=circle,draw,inner sep=0.1pt] (char) {#1};}}
\newcommand\mycircled[1]{\makebox[0pt]{\circled{#1}}}
\newcommand\mycircledx[1]{\makebox[0pt]{$\xcancel{\circled{#1}}$}}

\begin{example}
Let $\lambda=(3,2,2)$. Then 
$$P=([1,3,4,2],[1,3,5,2,4],[3,5,1,2,4],[4,5,1,2,3])$$
 is in $\mathrm{Path}_{\lambda}([1,3,4,2],[4,5,1,2,3])$. This corresponds to the following diagram.
\begin{center}
\begin{tabular}{cccc}
1&\multicolumn{1}{|c}{\mycircled{1}}&3&4\\
3&\multicolumn{1}{|c}{\mycircled{3}}&5&\mycircled{5}\\
\cline{3-4}
4&\multicolumn{1}{|c|}{5}&1&1\\
\cline{2-2}
2&2&2&2\\
5&4&4&3\\
\multicolumn{4}{c}{\makebox[0pt]{$(y_1-z_1)(y_3-z_1)(y_5-z_3)$}}
\end{tabular}
\end{center}
$\mathrm{weight}_{P,\lambda}(y;z)$ is written beneath the diagram. $(y_1-z_1)(y_3-z_1)$ corresponds to the circled $1$ and $3$ in column $1$ and $y_5-z_3$ corresponds to the circled $5$ in column $3$.
\end{example}

\section{The first formula}

\begin{definition}[$d_{u,\lambda}^w(y;z)$, $e_{uv}^w(y;z)$]
Given permutations $u,w\in S_\infty$ and a partition $\lambda$ we define a polynomial $d_{u,\lambda}^w(y;z)$ with nonnegative integer coefficients in terms of products of linear terms of the form $y_i-z_j$ as follows:
$$d_{u,\lambda}^w(y;z)=\sum_{P\in\mathrm{Path}_\lambda(u,w)}\mathrm{weight}_{P,\lambda}(y;z)$$
We define coefficients $e_{uv}^w(y;z)$ as follows. If $u,v,w\in S_\infty$, define
$$e_{uv}^w(y;z)=0$$
unless $\ell(wv^{-1}\monom{v})=\ell(w)+\ell(v^{-1}\monom{v})$, in which case define
$$e_{uv}^w(y;z)=d_{u,\lambda(v)}^{wv^{-1}\monom{v}}(y;z)$$
\end{definition}

\begin{theorem}\label{theorem:main}
Let $u,v\in S_\infty$ be such that there exists a $p>0$ for which $\ell(us_i)>\ell(u)$ for all $i<p$ and $\ell(vs_i)>\ell(v)$ for all $i>p$. Then
$$\sch_u(x;y)\sch_v(x;z)=\sum_{w\in S_\infty}e_{uv}^w(y;z)\sch_w(x;y)$$
\end{theorem}

The condition on $u$ and $v$ in the statement of Theorem \ref{theorem:main} is referred to as $u$ and $v$ having ``separated descents'' in \cite{kzj}, for which a puzzle rule is found for the coefficients $c_{uv}^w(0;0)$, a case that is also covered in \cite{Huang_2022}. In a more recent paper, \cite{knutson_separated} extends this to equivariant K-theory. The main result of \cite{groth_puzzle} is a similar formula for double Grothendieck polynomials, which also gives a positive formula for the $z=0$ case in the coefficients we consider.

We illustrate this with some examples.

\begin{example}
We compute the example using factorial Schur functions in \cite[6.4]{knutson2003puzzles}. Let $u=v=[1,3,2]$, and we first let $w=[1,3,2]$. Then
\begin{align*}
\monom{v}&=[3,1,2]\\
\code(\monom{v})&=(2)\\
\lambda(v)&=(1,1)\\
v^{-1}\monom{v}&=[2,1]\\
wv^{-1}\monom{v}&=[3,1,2]
\end{align*}
Then $c_{uv}^w(y;z)=(y_3-z_2)+(y_1-z_1)$, as witnessed by the following paths.
\begin{center}
\begin{tabular}{ccc}
1&\multicolumn{1}{|c}{3}&\mycircled{3}\\
\cline{2-3}
3&1&1\\
2&2&2\\
\multicolumn{3}{c}{$y_{3} - z_{2}$}
\end{tabular}
\hspace{5pt}
\begin{tabular}{ccc}
1&\multicolumn{1}{|c}{\mycircled{1}}&3\\
\cline{2-3}
3&3&1\\
2&2&2\\
\multicolumn{3}{c}{$y_{1} - z_{1}$}
\end{tabular}
\end{center}
The reference computes the result as $(y_3-z_1)+(y_1-z_2)$ instead, which of course yields the same result but differs from our formula in the terms produced by the combinatorial elements. As we show in Theorem \ref{theorem:equivpositive}, our formula has the advantage that it yields a positive formula in equivariant cohomology when $y$ is substituted for $z$; this is not the case for $(y_3-y_1)+(y_1-y_2)$, as $y_1-y_2$ is not a positive term in the sense of \cite{graham2001positivity}. Our formula gives $(y_3-y_2)+(y_1-y_1)$, and both terms are nonnegative. We note that Molev in \cite{molev2009littlewood} provides a positive formula stemming from Molev and Sagan's rule.

For the remaining paths we write $w$ at the top of the diagram to save space.
\begin{center}
\begin{tabular}{ccc}
\multicolumn{3}{c}{\makebox[0pt]{$[1,4,2,3]$}}\\
1&\multicolumn{1}{|c}{3}&4\\
\cline{2-3}
3&1&1\\
2&2&2\\
4&4&3\\
\multicolumn{3}{c}{$1$}
\end{tabular}
\hspace{5pt}
\begin{tabular}{ccc}
\multicolumn{3}{c}{\makebox[0pt]{$[2,3,1]$}}\\
1&\multicolumn{1}{|c}{2}&3\\
\cline{2-3}
3&3&2\\
2&1&1\\
\multicolumn{3}{c}{$\phantom{1}$}\\
\multicolumn{3}{c}{$1$}
\end{tabular}
\end{center}
Thus
$$\sch_{[1,3,2]}(x;y)\sch_{[1,3,2]}(x;z)=((y_3-z_2)+(y_1-z_1))\sch_{[1,3,2]}(x;y)+\sch_{[1,4,2,3]}(x;y)+\sch_{[2,3,1]}(x;y)$$
which agrees with \cite{knutson2003puzzles}.
\end{example}

\begin{example}
We compute the product of two factorial Schur polynomials with different numbers of $x$ variables, namely $s_{(2,1)}(x_1,x_2,x_3;y)s_{(2)}(x_1,x_2;z)$. This is the product $\sch_{[1,3,5,2,4]}(x;y)\sch_{[1,4,2,3]}(x;z)$. Notice that in this level of generality our formula cannot compute the product in the opposite order.
\begin{align*}
\monom{v}&=[4,1,2,3]\\
\code(\monom{v})&=(3)\\
\lambda(v)&=(1,1,1)\\
v^{-1}\monom{v}&=[2,1]
\end{align*}
\begin{center}
\begin{tabular}{cccc}
\multicolumn{4}{c}{\makebox[0pt]{$[1,3,5,2,4]$}}\\
1&\multicolumn{1}{|c}{3}&\mycircled{3}&\mycircled{3}\\
\cline{2-4}
3&1&1&1\\
5&5&5&5\\
2&2&2&2\\
4&4&4&4\\
\multicolumn{4}{c}{\makebox[0pt]{$(y_3-z_2)(y_3-z_3)$}}
\end{tabular}
$\phantom{(y_3-z_3)}$
\begin{tabular}{cccc}
\multicolumn{4}{c}{\makebox[0pt]{$[1,3,5,2,4]$}}\\
1&\multicolumn{1}{|c}{\mycircled{1}}&3&\mycircled{3}\\
\cline{2-4}
3&3&1&1\\
5&5&5&5\\
2&2&2&2\\
4&4&4&4\\
\multicolumn{4}{c}{\makebox[0pt]{$(y_{1} - z_{1})(y_{3} - z_{3})$}}
\end{tabular}
$\phantom{(y_3-z_3)}$
\begin{tabular}{cccc}
\multicolumn{4}{c}{\makebox[0pt]{$[1,3,5,2,4]$}}\\
1&\multicolumn{1}{|c}{\mycircled{1}}&\mycircled{1}&3\\
\cline{2-4}
3&3&3&1\\
5&5&5&5\\
2&2&2&2\\
4&4&4&4\\
\multicolumn{4}{c}{\makebox[0pt]{$(y_{1} - z_{1})(y_{1} - z_{2})$}}
\end{tabular}
\end{center}
$$((y_3-z_2)(y_3-z_3)+(y_1-z_1)(y_3-z_3)+(y_1-z_1)(y_1-z_2))\sch_{[1,3,5,2,4]}(x;y)$$
\begin{center}
\begin{tabular}{cccc}
\multicolumn{4}{c}{\makebox[0pt]{$[1,4,5,2,3]$}}\\
1&\multicolumn{1}{|c}{3}&\mycircled{3}&4\\
\cline{2-4}
3&1&1&1\\
5&5&5&5\\
2&2&2&2\\
4&4&4&3\\
\multicolumn{4}{c}{$y_{3} - z_{2}$}
\end{tabular}
\hspace{5pt}
\begin{tabular}{cccc}
\multicolumn{4}{c}{\makebox[0pt]{$[1,4,5,2,3]$}}\\
1&\multicolumn{1}{|c}{3}&4&\mycircled{4}\\
\cline{2-4}
3&1&1&1\\
5&5&5&5\\
2&2&2&2\\
4&4&3&3\\
\multicolumn{4}{c}{$y_{4} - z_{3}$}
\end{tabular}
\hspace{5pt}
\begin{tabular}{cccc}
\multicolumn{4}{c}{\makebox[0pt]{$[1,4,5,2,3]$}}\\
1&\multicolumn{1}{|c}{\mycircled{1}}&3&4\\
\cline{2-4}
3&3&1&1\\
5&5&5&5\\
2&2&2&2\\
4&4&4&3\\
\multicolumn{4}{c}{$y_{1} - z_{1}$}
\end{tabular}
\end{center}
$$((y_3-z_2)+(y_4-z_3)+(y_1-z_1))\sch_{[1,4,5,2,3]}(x;y)$$
\begin{center}
\begin{tabular}{cccc}
\multicolumn{4}{c}{\makebox[0pt]{$[2,3,5,1,4]$}}\\
1&\multicolumn{1}{|c}{\mycircled{1}}&2&3\\
\cline{2-4}
3&3&3&2\\
5&5&5&5\\
2&2&1&1\\
4&4&4&4\\
\multicolumn{4}{c}{$y_{1} - z_{1}$}
\end{tabular}
\hspace{5pt}
\begin{tabular}{cccc}
\multicolumn{4}{c}{\makebox[0pt]{$[2,3,5,1,4]$}}\\
1&\multicolumn{1}{|c}{2}&\mycircled{2}&3\\
\cline{2-4}
3&3&3&2\\
5&5&5&5\\
2&1&1&1\\
4&4&4&4\\
\multicolumn{4}{c}{$y_{2} - z_{2}$}
\end{tabular}
\hspace{5pt}
\begin{tabular}{cccc}
\multicolumn{4}{c}{\makebox[0pt]{$[2,3,5,1,4]$}}\\
1&\multicolumn{1}{|c}{2}&3&\mycircled{3}\\
\cline{2-4}
3&3&2&2\\
5&5&5&5\\
2&1&1&1\\
4&4&4&4\\
\multicolumn{4}{c}{$y_{3} - z_{3}$}
\end{tabular}
\end{center}
$$((y_1-z_1)+(y_2-z_2)+(y_3-z_3))\sch_{[2,3,5,1,4]}(x;y)$$
\begin{center}
\begin{tabular}{cccc}
\multicolumn{4}{c}{\makebox[0pt]{$[1,5,3,2,4]$}}\\
1 &\multicolumn{1}{|c}{\mycircled{1}}&3 &5 \\
\cline{2-4}
3 &3 &1 &1 \\
5 &5 &5 &3 \\
2 &2 &2 &2 \\
4 &4 &4 &4 \\
\multicolumn{4}{c}{$y_{1} - z_{1}$}
\end{tabular}
\hspace{5pt}
\begin{tabular}{cccc}
\multicolumn{4}{c}{\makebox[0pt]{$[1,5,3,2,4]$}}\\
1 &\multicolumn{1}{|c}{3} &5 &\mycircled{5}\\
\cline{2-4}
3 &1 &1 &1 \\
5 &5 &3 &3 \\
2 &2 &2 &2 \\
4 &4 &4 &4 \\
\multicolumn{4}{c}{$y_{5} - z_{3}$}
\end{tabular}
\hspace{5pt}
\begin{tabular}{cccc}
\multicolumn{4}{c}{\makebox[0pt]{$[1,5,3,2,4]$}}\\
1 &\multicolumn{1}{|c}{3} &\mycircled{3}&5 \\
\cline{2-4}
3 &1 &1 &1 \\
5 &5 &5 &3 \\
2 &2 &2 &2 \\
4 &4 &4 &4 \\
\multicolumn{4}{c}{$y_{3} - z_{2}$}
\end{tabular}
\end{center}
$$((y_1-z_1)+(y_5-z_3)+(y_3-z_2))\sch_{[1,5,3,2,4]}(x;y)$$
\begin{center}
\begin{tabular}{cccc}
\multicolumn{4}{c}{\makebox[0pt]{$[1,5,4,2,3]$}}\\
1&\multicolumn{1}{|c}{3}&4&5\\
\cline{2-4}
3&1&1&1\\
5&5&5&4\\
2&2&2&2\\
4&4&3&3\\
\multicolumn{4}{c}{$1$}
\end{tabular}
\hspace{5pt}
\begin{tabular}{cccc}
\multicolumn{4}{c}{\makebox[0pt]{$[2,4,5,1,3]$}}\\
1&\multicolumn{1}{|c}{2}&3&4\\
\cline{2-4}
3&3&2&2\\
5&5&5&5\\
2&1&1&1\\
4&4&4&3\\
\multicolumn{4}{c}{$1$}
\end{tabular}
\hspace{5pt}
\begin{tabular}{cccc}
\multicolumn{4}{c}{\makebox[0pt]{$[2,5,3,1,4]$}}\\
1&\multicolumn{1}{|c}{2}&3&5\\
\cline{2-4}
3&3&2&2\\
5&5&5&3\\
2&1&1&1\\
4&4&4&4\\
\multicolumn{4}{c}{$1$}
\end{tabular}
\end{center}
\begin{center}
\begin{tabular}{cccc}
\multicolumn{4}{c}{\makebox[0pt]{$[1,6,3,2,4,5]$}}\\
1&\multicolumn{1}{|c}{3}&5&6\\
\cline{2-4}
3&1&1&1\\
5&5&3&3\\
2&2&2&2\\
4&4&4&4\\
6&6&6&5\\
\multicolumn{4}{c}{$1$}
\end{tabular}
\end{center}
$$\sch_{[1,5,4,2,3]}(x;y)+\sch_{[2,4,5,1,3]}(x;y)+\sch_{[2,5,3,1,4]}(x;y)+\sch_{[1,6,3,2,4,5]}(x;y)$$
\end{example}

\begin{example}
Our formula computes the product of the double Schubert polynomial $\sch_{[1,4,3,2]}(x;y)$, which is not a factorial Schur polynomial, with the factorial Schur polynomial $s_{(2,1)}(x_1,x_2;z)$. We set $u=[1,4,3,2]$, $v=[2,4,1,3]$. We compute the coefficient where $w=[3,4,1,2]$. We have
\begin{align*}
\monom v&=[4,2,1,3]\\
\code(\monom{v})&=(3,1)\\
\lambda(v)&=(2,1,1)\\
v^{-1}\monom{v}&=[2,1]\\
wv^{-1}\monom{v}&=[4,3,1,2]
\end{align*}
\begin{center}
\begin{tabular}{cccc}
1&\multicolumn{1}{|c}{\mycircled{1}}&3&4\\
\cline{3-4}
4&\multicolumn{1}{|c|}{\mycircled{4}}&4&3\\
\cline{2-2}
3&3&1&1\\
2&2&2&2\\
\multicolumn{4}{c}{\makebox[0pt]{$(y_{1} - z_{1})(y_{4} - z_{1})$}}
\end{tabular}
$\phantom{(y_3-z_2)}$
\begin{tabular}{cccc}
1&\multicolumn{1}{|c}{3}&\mycircled{3}&4\\
\cline{3-4}
4&\multicolumn{1}{|c|}{\mycircled{4}}&4&3\\
\cline{2-2}
3&1&1&1\\
2&2&2&2\\
\multicolumn{4}{c}{\makebox[0pt]{$(y_{3} - z_{2})(y_{4} - z_{1})$}}
\end{tabular}
$\phantom{(y_3-z_2)}$
\begin{tabular}{cccc}
1&\multicolumn{1}{|c}{3}&4&\mycircled{4}\\
\cline{3-4}
4&\multicolumn{1}{|c|}{\mycircled{4}}&3&3\\
\cline{2-2}
3&1&1&1\\
2&2&2&2\\
\multicolumn{4}{c}{\makebox[0pt]{$(y_{4} - z_{1})(y_{4} - z_{3})$}}
\end{tabular}
\end{center}
Thus
$$c_{uv}^w(y;z)=(y_1-z_1)(y_4-z_1)+(y_3-z_2)(y_4-z_1)+(y_4-z_1)(y_4-z_3)$$
Now we compute the coefficient for $w=[3,5,1,2,4]$. We have $wv^{-1}\monom{v}=[5,3,1,2,4]$.
\begin{center}
\begin{tabular}{cccc}
1&\multicolumn{1}{|c}{\mycircled{1}}&3&5\\
\cline{3-4}
4&\multicolumn{1}{|c|}{5}&5&3\\
\cline{2-2}
3&3&1&1\\
2&2&2&2\\
5&4&4&4\\
\multicolumn{4}{c}{$y_1-z_1$}
\end{tabular}
$\phantom{(y_3-z_2)}$
\begin{tabular}{cccc}
1&\multicolumn{1}{|c}{3}&\mycircled{3}&5\\
\cline{3-4}
4&\multicolumn{1}{|c|}{5}&5&3\\
\cline{2-2}
3&1&1&1\\
2&2&2&2\\
5&4&4&4\\
\multicolumn{4}{c}{$y_3-z_2$}
\end{tabular}
$\phantom{(y_3-z_2)}$
\begin{tabular}{cccc}
1&\multicolumn{1}{|c}{3}&4&5\\
\cline{3-4}
4&\multicolumn{1}{|c|}{\mycircled{4}}&3&3\\
\cline{2-2}
3&1&1&1\\
2&2&2&2\\
5&5&5&4\\
\multicolumn{4}{c}{\makebox[0pt]{$y_{4} - z_{1}$}}
\end{tabular}
$\phantom{(y_3-z_2)}$
\begin{tabular}{cccc}
1&\multicolumn{1}{|c}{3}&5&\mycircled{5}\\
\cline{3-4}
4&\multicolumn{1}{|c|}{5}&3&3\\
\cline{2-2}
3&1&1&1\\
2&2&2&2\\
5&4&4&4\\
\multicolumn{4}{c}{\makebox[0pt]{$y_{5} - z_{3}$}}
\end{tabular}
\end{center}
Thus
$$c_{uv}^w(y;z)=(y_1-z_1)+(y_3-z_2)+(y_4-z_1)+(y_5-z_3)$$
\end{example}

\section{Proof of Theorem \ref{theorem:main}}

We show in this section that if $u$ and $v$ satisfy the hypotheses of Theorem \ref{theorem:main}, then we may reduce multiplying $\sch_u(x;y)$ by $\sch_v(x;z)$ to multiplying $\sch_u(x;y)$ by the double Schubert polynomial corresponding to $v$'s dominant approximation, $\sch_{\monom v}(x;z)$. We then show that a double Schubert polynomial corresponding to a dominant permutation is a product of factorial elementary symmetric polynomials, which will allow us to reduce the computation to applying the Pieri formula. The special case of the Pieri formula that we will need is Proposition \ref{proposition:minipieri}, which we will prove in Section \ref{section:pieriproof}.

\begin{lemma} \label{lemma:nodescentzero}
Let $u,v,w\in S_\infty$, and suppose $i>0$ is such that $\ell(ws_i)<\ell(w)$. If $\ell(us_i)>\ell(u)$ and $\ell(vs_i)>\ell(v)$, then $c_{uv}^w(y;z)=0$.
\end{lemma}
\begin{proof}
Recall that
$$c_{uv}^w(x;z)=\partial_u^w(\sch_v(x;z))$$
Since $\ell(us_i)>\ell(u)$ we have that
$$\partial_u^w=\partial_u^{ws_i}\partial^{s_i}$$
Since $\partial^{s_i}(\sch_v(x;z))=0$, we have the result.
\end{proof}

\begin{lemma} \label{lemma:descentcoeffdown}
Let $u,v,w\in S_\infty$ and let $i>0$ be such that $\ell(us_i)>\ell(u)$, $\ell(vs_i)>\ell(v)$, and $\ell(ws_i)>\ell(w)$. Then
$$c_{uv}^w(y;z)=c_{u,vs_i}^{ws_i}(y;z)$$
\end{lemma}
\begin{proof}
We have that
$$\partial_u^{ws_i}=\partial_u^w\partial^{s_i}$$
Hence
$$\partial_u^{ws_i}(\sch_{vs_i}(x;z))=\partial_u^w(\partial^{s_i}(\sch_{vs_i}(x;z)))=\partial_u^w(\sch_v(x;z))$$
and the result follows.
\end{proof}

\begin{proposition} \label{proposition:premain}
Suppose $p$ is a positive integer, $u\in S_\infty$ satisfies $\ell(us_i)>\ell(u)$ for all $i<p$, and $v\in S_\infty$ satisfies $\ell(vs_i)>\ell(v)$ for all $i>p$. Then $c_{uv}^w(y;z)=0$ unless $\ell(wv^{-1}\monom{v})=\ell(w)+\ell(v^{-1}\monom{v})$, in which case
$$c_{uv}^w(y;z)=c_{u,\monom{v}}^{wv^{-1}\monom{v}}(y;z)$$
\end{proposition}
\begin{proof}
Note that for any permutation $v$ we have that $\ell(vs_i)>\ell(v)$ for all $i>p$ if and only if $\code_i(v)=0$ for all $i>p$. We prove the result by induction on $\ell(\monom v)-\ell(v)$. If $\ell(\monom v)-\ell(v)=0$, then $c_{uv}^w(y;z)=c_{u,\monom{v}}^{wv^{-1}\monom v}(y;z)$ because $v^{-1}\monom{v}=1$ and $v=\monom{v}$. Otherwise, suppose $i$ is the maximal index such that $\code_i(v)<\code_{i+1}(v)$. Since $\code_{i+1}(v)\neq 0$, we must have that $i<p$. Therefore, by assumption $\ell(us_i)>\ell(u)$. Also, since $\code_j(v)=\code_j(vs_i)=0$ for all $j>p$, it follows that $\ell(vs_is_j)>\ell(vs_i)$ for all $j>p$.

We split the argument into whether or not $\ell(ws_i)<\ell(w)$. Assume first that $\ell(ws_i)<\ell(w)$. Note that $v^{-1}\mu_v$ has a specific expression as a product of $\ell(\mu_v)-\ell(v)$ simple reflections (the ones that occur in the recursion passing from $v$ to $\mu_v$), with the first being $s_i$, and in order for $w(v^{-1}\mu_v)$ to have length $\ell(w)+\ell(v^{-1}\mu_v)$ each of these must successively increase the length by $1$ in the multiplication of $w$ on the right by $v^{-1}\mu_v$. Since the length decreases upon multiplication by the first reflection $s_i$, we must have that $\ell(wv^{-1}\monom{v})<\ell(w)+\ell(v^{-1}\monom{v})$. By Lemma \ref{lemma:nodescentzero} we also have that $c_{uv}^w(y;z)=0$ since both $\ell(us_i)>\ell(u)$ and $\ell(vs_i)>\ell(v)$. Thus, if $\ell(ws_i)<\ell(w)$ then $c_{uv}^w(y;z)=0$ and $\ell(wv^{-1}\mu_v)<\ell(w)+\ell(v^{-1}\mu_v)$, so the result follows in that case.

Assume then that $\ell(ws_i)>\ell(w)$. In that case, by Lemma \ref{lemma:descentcoeffdown}, $c_{uv}^w(y;z)=c_{u,vs_i}^{ws_i}(y;z)$. We have that $\mu_{vs_i}=\mu_v$ by definition, and $\ell(\mu_v)-\ell(vs_i)<\ell(\mu_v)-\ell(v)$. Note that
$$wv^{-1}\mu_v=w(s_is_i)v^{-1}\mu_v=ws_i(vs_i)^{-1}\mu_v$$
By the induction hypothesis, we have that $c_{u,vs_i}^{ws_i}(y;z)=0$ unless 
$$\ell(wv^{-1}\mu_v)=\ell(ws_i(vs_i)^{-1}\mu_v)=\ell(ws_i)+\ell((vs_i)^{-1}\mu_v)=\ell(w)+\ell(v^{-1}\mu_v)$$
In the case where the length condition is satisfied (where there is a possibility that $c_{u,vs_i}^{ws_i}(y;z)\neq 0$), by the induction hypothesis we have that
$$c_{uv}^w(y;z)=c_{u,vs_i}^{ws_i}(y;z)=c_{u,\monom{vs_i}}^{ws_i(vs_i)^{-1}\monom{vs_i}}(y;z)=c_{u,\monom{v}}^{wv^{-1}\monom{v}}(y;z)$$
and it follows that $c_{uv}^w(y;z)=0$ unless $\ell(wv^{-1}\mu_v)=\ell(w)+\ell(v^{-1}\mu_v)$, in which case $c_{uv}^w(y;z)=c_{u,\mu_v}^{wv^{-1}\mu_v}(y;z)$.

Having verified in the two exhaustive cases $\ell(ws_i)<\ell(w)$ and $\ell(ws_i)>\ell(w)$ that the result follows after the increase in length difference from the induction hypothesis, we have the result by induction.
\end{proof}


\begin{definition}[Factorial elementary symmetric polynomials $E_{k}(x;z)$, $E_{p,k}(x;z)$]
Let $E_k(x;z)$ be the factorial elementary symmetric polynomial defined by
$$E_{k}(x;z)=\prod_{i=1}^k (x_i-z_1)$$
This is the double Schubert polynomial $\sch_{c_{k,k}}(x;z)$ where
$$c_{p,k}=s_{k+1-p}s_{k+2-p}\cdots s_k$$
We define more generally 
$$E_{p,k}(x;z)=\sch_{c_{p,k}}(x;z)$$
so that $E_k(x;z)=E_{k,k}(x;z)$. If $p<0$ or $p>k$, we define
$$E_{p,k}(x;z)=0$$
Since exactly one $z$ variable occurs in the product for $E_k(x;z)$, namely $z_1$, we use the notation $E_k(x;z_j)$ for any $j$ to substitute $z_j$ for $z_1$.
\end{definition}

\begin{proposition} \label{proposition:minipieri}
Let $j,k$ be positive integers and suppose $u\in S_\infty$. Then
$$\sch_u(x;y)E_k(x;z_j)=\sum_{u\tom{k} w}\mathrm{weight}_{u,k}^w(y;z_j)\sch_w(x;y)$$
\end{proposition}

We postpone the proof of this to Section \ref{section:pieriproof}.


The next lemma, for which we cite \cite[(6.14)]{notes}, expresses the double Schubert polynomial corresponding to a dominant permutation as a product of these factorial elementary symmetric polynomials.

\begin{lemma} \label{lemma:monomial}
Suppose $v\in S_\infty$. Then
$$\sch_{\monom{v}}(x;z)=\prod_{i=1}^\infty E_{\lambda_i(v)}(x;z_i)$$
\end{lemma}
\begin{proof}
In \cite[(6.14)]{notes} there is a formula for the double Schubert polynomial corresponding to a dominant permutation $\mu$ in terms of the coordinates of the boxes in the Young diagram of $\code(\mu)$. The formula is
$$\sch_{\monom v}(x;z)=\prod_{(i,j)\in Y_{\code(\monom v)}}(x_i-z_j)$$
If we fix $j$ and consider the terms involving $z_j$ in this product, the terms that occur are $x_i-z_j$ where $1\leq i\leq \lambda_j(v)$, since the length of this column in the diagram is the value at the corresponding index in the conjugate partition of the code. The product of these terms is therefore equal to $E_{\lambda_j(v)}(x;z_j)$. Iterating over $j$, we obtain the result.
\end{proof}

From this we may derive the following. 

\begin{proposition} \label{proposition:mulmonomial}
Suppose $u,v\in S_\infty$. Then
$$\sch_u(x;y)\sch_{\monom{v}}(x;z)=\sum_{w\in S_\infty} d_{u,\lambda(v)}^w(y;z)\sch_w(x;y)$$
\end{proposition}
\begin{proof}
Assume the length of the partition $\lambda(v)$ is $m$. To shorten notation, for an integer $0\leq p\leq m$ define a partition $\lambda^p$ as
$$\lambda^p=(\lambda_1(v),\ldots,\lambda_p(v))$$
We note by Lemma \ref{lemma:monomial} that
$$\sch_{\monom{v}}(x;z)=\prod_{i=1}^m E_{\lambda_i(v)}(x;z_i)$$
We apply Proposition \ref{proposition:minipieri} and induction to successively multiply by the factors of $\sch_{\mu_v}(x;z)$. The statement we will prove is that for $0\leq p\leq m$ we have
$$\sch_u(x;y)E_{\lambda_1(v)}(x;z_1)\cdots E_{\lambda_p(v)}(x;z_p)=\sum_w d_{u,\lambda^p}^w(y;z)\sch_w(x;y)$$
For the base case, multiplying by none of the factors, the product is $\sch_u(x;y)$. Assume the induction hypothesis that for some $p$ with $p-1<m$ we have
$$\sch_u(x;y)E_{\lambda_1(v)}(x;z_1)\cdots E_{\lambda_{p-1}(v)}(x;z_{p-1})=\sum_w d_{u,\lambda^{p-1}}^w(y;z)\sch_w(x;y)$$
We multiply both sides by the additional factor $E_{\lambda_p(v)}(x;z_p)$, applying Proposition \ref{proposition:minipieri}:
$$\sch_u(x;y)E_{\lambda_1(v)}(x;z_1)\cdots E_{\lambda_{p}(v)}(x;z_{p})=\sum_w\sum_{w':w\tom{\lambda_p(v)} w'} d_{u,\lambda^{p-1}}^w(y;z)\mathrm{weight}_{w,\lambda_p(v)}^{w'}(y;z_p)\sch_{w'}(x;y)$$
Interchange the order of summation on the right hand side to obtain that the coefficient of $\sch_{w'}(x;y)$ is, fixing $w'$ and summing over all $w$,
$$\sum_{w:w\tom{\lambda_p(v)} w'}d_{u,\lambda^{p-1}}^w(y;z)\mathrm{weight}_{w,\lambda_p(v)}^{w'}(y;z_p)$$
We claim that this is equal to $d_{u,\lambda^p}^{w'}(y;z)$. Recall that $d_{u,\lambda^{p-1}}^w(y;z)$ is a sum over paths 
$$u=u_0\tom{\lambda_1(v)} u_1\tom{\lambda_2(v)}\cdots\tom{\lambda_{p-1}(v)} u_{p-1}=w$$
of the product
$$\prod_{i=1}^{p-1}\mathrm{weight}_{u_{i-1},\lambda_i(v)}^{u_i}(y;z_i)$$
We are extending each path by an additional edge $w\tom{\lambda_p(v)} w'$. As this ranges over all $w$ (the penultimate element in the path), we obtain exactly the paths that occur in the sum for $d_{u,\lambda^p}^{w'}(y;z)$ with the desired extra factor occurring in the weight, obtaining the result by induction.
\end{proof}

\begin{proof}[Proof of Theorem \ref{theorem:main}]
By Proposition \ref{proposition:premain}, we have under the conditions that $c_{uv}^w(y;z)=0$ unless $\ell(wv^{-1}\mu_v)=\ell(w)+\ell(v^{-1}\mu_v)$, in which case we have
$$c_{uv}^w(y;z)=c_{u,\monom{v}}^{wv^{-1}\monom{v}}(y;z)$$
By Proposition \ref{proposition:mulmonomial} we have that
$$\sch_{u}(x;y)\sch_{\monom{v}}(x;z)=\sum_{w'\in S_\infty}{d_{u,\lambda(v)}^{w'}(y;z)\sch_{w'}(x;y)}$$
Picking out the coefficient of $\sch_{w'}(x;y)$ with $w'=wv^{-1}\monom{v}$ we obtain the result.
\end{proof}

\section{Positive formula for (skew) double Schubert polynomials}

Theorem \ref{theorem:main} yields a nontrivial formula even when $u$ is the identity. In that case, $\partial_u^w=\partial^w$, hence $c_{uv}^w(x;y)=\partial^w(\sch_v(x;y))$. Thus 
$$c_{1,v}^{1}(x;y) = \sch_v(x;y)$$
and hence our formula yields a positive formula for double Schubert polynomials. We record this as a theorem.

\begin{theorem} \label{theorem:schubertformula}
For all $v\in S_\infty$ we have
$$\sch_v(x;y)=e_{1,v}^1(x;y)$$
and hence $\sch_v(x;y)$ is a polynomial in the differences $x_i-y_j$ with nonnegative integer coefficients.
\end{theorem}

\begin{example} \label{example:schubpoly}
We use the formula to obtain $\sch_{[1,4,3,2]}(x;y)$. We set $u = 1$, $v=[1,4,3,2]$, and $w=1$.
Then
\begin{align*}
\monom{v}&=[4,3,1,2]\\
\code(\monom v)&=(3,2)\\
\lambda(v)&=(2,2,1)\\
v^{-1}\monom{v}&=[2,3,1]\\
wv^{-1}\monom{v}&=[2,3,1]
\end{align*}
We iterate over the paths to obtain the polynomial.
\begin{center}
\begin{tabular}{cccc}
1&\multicolumn{1}{|c}{\mycircled{1}}&\mycircled{1}&2\\
\cline{4-4}
2&\multicolumn{1}{|c}{\mycircled{2}}&\multicolumn{1}{c|}{3}&3\\
\cline{2-3}
3&3&2&1\\
\multicolumn{4}{c}{\makebox[0pt]{$(x_{1} - y_{1})(x_{1} - y_{2})(x_{2} - y_{1})$}}
\end{tabular}
$\phantom{(x_{1} - y_{1})(x_1-y_2)}$
\begin{tabular}{cccc}
1&\multicolumn{1}{|c}{\mycircled{1}}&2&\mycircled{2}\\
\cline{4-4}
2&\multicolumn{1}{|c}{\mycircled{2}}&\multicolumn{1}{c|}{3}&3\\
\cline{2-3}
3&3&1&1\\
\multicolumn{4}{c}{\makebox[0pt]{$(x_{1} - y_{1})(x_{2} - y_{1})(x_{2} - y_{3})$}}
\end{tabular}
\end{center}

\begin{center}
\begin{tabular}{cccc}
1&\multicolumn{1}{|c}{\mycircled{1}}&\mycircled{1}&2\\
\cline{4-4}
2&\multicolumn{1}{|c}{3}&\multicolumn{1}{c|}{\mycircled{3}}&3\\
\cline{2-3}
3&2&2&1\\
\multicolumn{4}{c}{\makebox[0pt]{$(x_{1} - y_{1})(x_{1} - y_{2})(x_{3} - y_{2})$}}
\end{tabular}
$\phantom{(x_{1} - y_{1})(x_1-y_2)}$
\begin{tabular}{cccc}
1&\multicolumn{1}{|c}{\mycircled{1}}&2&\mycircled{2}\\
\cline{4-4}
2&\multicolumn{1}{|c}{3}&\multicolumn{1}{c|}{\mycircled{3}}&3\\
\cline{2-3}
3&2&1&1\\
\multicolumn{4}{c}{\makebox[0pt]{$(x_{1} - y_{1})(x_{2} - y_{3})(x_{3} - y_{2})$}}
\end{tabular}
\end{center}

\begin{center}
\begin{tabular}{cccc}
1&\multicolumn{1}{|c}{2}&\mycircled{2}&\mycircled{2}\\
\cline{4-4}
2&\multicolumn{1}{|c}{3}&\multicolumn{1}{c|}{\mycircled{3}}&3\\
\cline{2-3}
3&1&1&1\\
\multicolumn{4}{c}{\makebox[0pt]{$(x_{2} - y_{2})(x_{2} - y_{3})(x_{3} - y_{2})$}}
\end{tabular}
\end{center}
\end{example}

Thus
\begin{align*}
\sch_{[1,4,3,2]}(x;y)&=(x_{1} - y_{1})(x_{1} - y_{2})(x_{2} - y_{1})+(x_{1} - y_{1})(x_{2} - y_{1})(x_{2} - y_{3})+(x_{1} - y_{1})(x_{1} - y_{2})(x_{3} - y_{2})\\
&+(x_{1} - y_{1})(x_{2} - y_{3})(x_{3} - y_{2})+(x_{2} - y_{2})(x_{2} - y_{3})(x_{3} - y_{2})
\end{align*}

While our formula for double Schubert polynomials derived in this way is new, it is not the first positive formula for double Schubert polynomials in terms of $x_i-y_j$. For example, see the formula in \cite{knutsonpipe} using the pipe dreams (or equivalently RC-graphs) in \cite{bbrc}. We suspect from empirical evidence, though have not proved, that a modification of our formula for $\sch_v(x;y)$ to computing $c_{1,w_0(n)}^{v^{-1}w_0(n)}(x;y)$ but applying Proposition \ref{proposition:minipieri} in increasing order of degree instead of decreasing order yields the same terms as the definition of pipe dream polynomials with the corresponding RC-graph obtained directly by sorting the columns of our diagram. It would be interesting to establish this correspondence in general. We give an example.

\begin{example}[Illustration of the conjectural relation to RC-graphs] \label{example:rc}
An example RC-graph for the permutation $v=[3,1,4,6,5,2]$ is given in Figure (3.1) in \cite{bbrc}. Namely, the example is
\begin{center}
\begin{tabular}{ccccccc}
&1&2&3&4&5&6\\
1&$+$&$+$&$\cdot$&$\cdot$&$+$&$\cdot$\\
2&$\cdot$&$+$&$\cdot$&$\cdot$&$\cdot$&\\
3&$\cdot$&$+$&$\cdot$&$\cdot$&&\\
4&$\cdot$&$\cdot$&$\cdot$&&&\\
5&$+$&$\cdot$&&&&\\
6&$\cdot$&&&&&
\end{tabular}
\end{center}
The corresponding term in the double Schubert polynomial given by the formula in \cite{knutsonpipe} (each plus sign gives a factor of $x_r-y_c$, where $r$ is the row and $c$ is the column) would be
$$(x_{1} - y_{1})(x_{1} - y_{2})(x_{1} - y_{5})(x_{2} - y_{2})(x_{3} - y_{2})(x_{5} - y_{1})$$
Applying Proposition \ref{proposition:minipieri} to 
$$\sch_{w_0(5)}(x;y)=E_1(x;y_5)E_2(x;y_4)E_3(x;y_3)E_4(x;y_2)E_5(x;y_1)$$
to compute $c_{1,w_0(5)}^{v^{-1}w_0(5)}(x;y)$ we obtain for one of the terms the following diagram, multiplying in increasing order of degree, beginning the path in the diagram from the right:
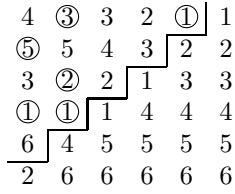
\begin{center}
\begin{tabular}{cccccc}
4&\mycircled{3}&3&2&\multicolumn{1}{c|}{\mycircled{1}}&1\\
\cline{5-5}
\mycircled{5}&5&4&\multicolumn{1}{c|}{3}&2&2\\
\cline{4-4}
3&\mycircled{2}&\multicolumn{1}{c|}{2}&1&3&3\\
\cline{3-3}
\mycircled{1}&\multicolumn{1}{c|}{\mycircled{1}}&1&4&4&4\\
\cline{2-2}
\multicolumn{1}{c|}{6}&4&5&5&5&5\\
\cline{1-1}
2&6&6&6&6&6
\end{tabular}
\end{center}
This diagram yields the same term as the RC-graph. Sorting the columns in this diagram, we obtain 
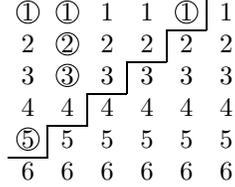
\begin{center}
\begin{tabular}{cccccc}
\mycircled{1}&\mycircled{1}&1&1&\multicolumn{1}{c|}{\mycircled{1}}&1\\
\cline{5-5}
2&\mycircled{2}&2&\multicolumn{1}{c|}{2}&2&2\\
\cline{4-4}
3&\mycircled{3}&\multicolumn{1}{c|}{3}&3&3&3\\
\cline{3-3}
4&\multicolumn{1}{c|}{4}&4&4&4&4\\
\cline{2-2}
\multicolumn{1}{c|}{\mycircled{5}}&5&5&5&5&5\\
\cline{1-1}
6&6&6&6&6&6
\end{tabular}
\end{center}
which coincides with the RC-graph. The interested reader may wish to verify that the remaining $14$ diagrams all coincide with other RC-graphs for $v$ in the same way.

While one might initially expect that any formula should involve the same term at some point, this specific term does not occur at all in the formula for $e_{1,v}^1(x;y)$, which uses a different dominant permutation and uses multiplication in decreasing order of degree (with increasing order of indices). Expressions of double Schubert polynomials $\sch_v(x;y)$ for permutations $v$ as polynomials in $x_i-y_j$ are generally not at all unique. This is easily seen in the simple example of the polynomial
$$\sch_{s_2}(x;y)=(x_1-y_1)+(x_2-y_2)=(x_1-y_2)+(x_2-y_1)$$
and in general for the linear double Schubert polynomials in $n$ variables there are $n!$ ways to group the terms.
\end{example}

\hspace{12pt}

Kirillov in \cite{kirillov2007skew} defines skew Schubert polynomials as polynomials of the form $u^{-1}\partial_u^w(\sch_{w_0(n)}(x))$, or equivalently $u^{-1}(c_{u,w_0(n)}^w(x;0))$, and conjectures that they have nonnegative integer coefficients. We note that there are other notions of skew Schubert polynomials, for example in \cite{chen2004skew}, \cite{bsskew}, \cite{lenart2003skew}, and \cite{tamvakis2020tableau}. Using our formulas, we have been able to prove Kirillov's conjecture.

\begin{theorem} \label{theorem:kirpositive}
The skew Schubert polynomials defined by Kirillov in \cite{kirillov2007skew} have nonnegative integer coefficients.
\end{theorem}
\begin{proof}
We have that $c_{u,w_0(n)}^w(x;0)=e_{u,w_0(n)}^w(x;0)$ since $w_0(n)$ is dominant, hence any permutation of this polynomial has nonnegative integer coefficients.
\end{proof}

The polynomials $c_{u,w_0(n)}^w(0;-y)$ seem to be closely related to the skew Schubert polynomials in \cite{lenart2003skew} and \cite{bsskew} in that they are nonnegative linear combinations of ordinary Schubert polynomials via Littlewood-Richardson coefficients, and it would be worthwhile to investigate the connection between them, which we will not do here. The ``double skew Schubert polynomials'' $c_{u,w_0(n)}^w(x;-y)$ would then seem to connect Kirillov's skew Schubert polynomials and polynomials similar to Lenart's/Bergeron's/Sottile's.

\begin{example}
We compute, using our positive formula, the skew Schubert polynomial computed in Example 3 of \cite{kirillov2007skew}, up to a permutation. Specifically, the polynomial being computed is $c_{u,w_0(3)}^w(x;0)$, where $u=[3,1,2,4]$ and $w=[4,3,1,2]$.
\begin{center}
\begin{tabular}{cccc}
3 &\multicolumn{1}{|c}{\mycircled{3}}&\mycircled{3}&4 \\
\cline{4-4}
1 &\multicolumn{1}{|c}{\mycircled{1}}&\multicolumn{1}{c|}{4} &3 \\
\cline{3-3}
2 &\multicolumn{1}{|c|}{4} &1 &1 \\
\cline{2-2}
4&2 &2 &2 \\
\multicolumn{4}{c}{\makebox[0pt]{$(x_{1} - y_{1})(x_{3} - y_{1})(x_{3} - y_{2})$}}
\end{tabular}
$\phantom{(x_1-y_1)(x_2-y_1)}$
\begin{tabular}{cccc}
3 &\multicolumn{1}{|c}{\mycircled{3}}&4 &\mycircled{4}\\
\cline{4-4}
1 &\multicolumn{1}{|c}{\mycircled{1}}&\multicolumn{1}{c|}{3} &3 \\
\cline{3-3}
2 &\multicolumn{1}{|c|}{4} &1 &1 \\
\cline{2-2}
4&2 &2 &2 \\
\multicolumn{4}{c}{\makebox[0pt]{$(x_{1} - y_{1})(x_{3} - y_{1})(x_{4} - y_{3})$}}
\end{tabular}
\end{center}
\begin{center}
\begin{tabular}{cccc}
3 &\multicolumn{1}{|c}{4} &\mycircled{4}&\mycircled{4}\\
\cline{4-4}
1 &\multicolumn{1}{|c}{\mycircled{1}}&\multicolumn{1}{c|}{3} &3 \\
\cline{3-3}
2 &\multicolumn{1}{|c|}{3} &1 &1 \\
\cline{2-2}
4&2 &2 &2 \\
\multicolumn{4}{c}{\makebox[0pt]{$(x_{1} - y_{1})(x_{4} - y_{2})(x_{4} - y_{3})$}}
\end{tabular}
\end{center}
Setting $y=0$, we obtain $(x_3^2+x_3x_4+x_4^2)x_1$, which is the desired permutation of the result given in the article, $(x_1^2+x_1x_4+x_4^2)x_2$. 
\end{example}
\section{Equivariant positivity}

The order of multiplication by factorial elementary symmetric polynomials in computation of $e_{uv}^w(y;z)$ may seem arbitrary, and indeed any order will of course give the same overall result. Example \ref{example:rc} illustrates that different choices can result in interesting combinatorial consequences. The particular choice of multiplying in decreasing order of degree (with increasing order of indices) is advantageous in that it gives a positive formula for $c_{uv}^w(y;y)$ for the relevant $u$ and $v$, as we show in this section.

\begin{definition}[Graham-nonnegative]
The negative roots are the linear polynomials of the form $y_j-y_i$ with $j>i$, and the simple negative roots are the polynomials $\alpha_i=y_{i+1}-y_i$. It is easy to see that the negative roots have nonnegative coefficients when expressed in terms of the simple negative roots. A polynomial $P(y)$ is said to be Graham-nonnegative if $P(y)$ can be expressed as a polynomial in the simple negative roots with nonnegative integer coefficients.
\end{definition}

\begin{theorem} \label{theorem:equivpositive}
Let $u,w\in S_\infty$, let $\lambda$ be a partition, and suppose $P\in\mathrm{Path}_\lambda(u,w)$. Then $\mathrm{weight}_{P,\lambda}(y;y)$ is Graham-nonnegative.
\end{theorem}

This requires some setup. 

\begin{definition}[Factors of a path]
Let $\lambda$ be a partition and let $P=(u_0,\ldots,u_m)\in \mathrm{Path}_\lambda(u,w)$. Then if $a,q$ are integers we say that $(a,q)$ is a \emph{factor} of $P$ if $1\leq q\leq m$, $1\leq a\leq \lambda_q$, and 
$$u_{q-1}(a)=u_q(a)$$
This factor is said to be in row $a$ and column $q$. If $(a,q)$ is a factor of $P$, then $(a,q)$ is said to be a \emph{negative factor} if $u_q(a)<q$, a \emph{zero factor} if $u_q(a)=q$, and a \emph{positive factor} if $u_q(a)>q$.
\end{definition}

\begin{lemma} \label{lemma:pathequal}
Let $u,w$ be permutations, let $\lambda$ be a partition, and let $P\in \mathrm{Path}_\lambda(u,w)$.  If $(j,k)$ is a negative factor of $P$, then there exists a zero factor $(a,q)$ of $P$ with $q\leq k$.
\end{lemma}
\begin{proof}
Let $P=(u_0,\ldots,u_m)$. We note that it is a consequence of the definition of $\tom{}$ that for each $1\leq q\leq m$ we have that if $j\leq \lambda_q$ then $u_{q-1}(j)\leq u_q(j)$. We prove by induction on $q$ that if there are no integers $j,k$ with $k\leq q$ such that $(j,k)$ is a zero factor of $P$, then for all factors $(j,k)$ of $P$ with $1\leq k\leq q$ we have that $(j,k)$ is a positive factor of $P$. For $q=1$ this is clear, as in that case any factor is either a positive factor or a zero factor. Otherwise, suppose it holds for all $q'<q$ for some $q>1$. Then if there are no zero factors to the left of column $q$, since $u_{q-1}(j)\geq q$ for all $j\leq \lambda_{q-1}$ by the induction hypothesis we have that $u_q(j)\geq q$ for all $j\leq \lambda_q$. If $u_q(a)=q$ for some $a\leq\lambda_q$ then $u_q(a)=u_{q-1}(a)=q$ and it follows that $(a,q)$ is a zero factor of $P$. Otherwise we have that $u_q(j)>q$ for all $j\leq \lambda_q$, so that there can only be positive factors of $P$ in column $q$ and the result follows by induction.
\end{proof}

\begin{proof}[Proof of Theorem \ref{theorem:equivpositive}]
Let $u,w\in S_\infty$, let $\lambda$ be a partition, and let $P=(u_0,\ldots,u_m)\in\mathrm{Path}_\lambda(u,w)$. We show that $\mathrm{weight}_{P,\lambda}(y;y)$ is Graham-nonnegative. Recall the definition
$$\mathrm{weight}_{P,\lambda}(y;y)=\prod_{i=1}^m\mathrm{weight}_{u_{i-1},\lambda_i}^{u_i}(y;y_i)$$
In this product, for each $j$ the factor $\mathrm{weight}_{u_{j-1},\lambda_j}^{u_j}(y;y_j)$ can only fail to be Graham-nonnegative if there exists at least one negative factor $(i,j)$ of $P$, which would contribute $y_{u_{j}(i)}-y_j$ to the product with $u_j(i)<j$. However, by Lemma \ref{lemma:pathequal}, if there is such a factor then there exists a zero factor $(a,q)$ with $q\leq j$; this contributes $y_q-y_q=0$ to $\mathrm{weight}_{u_{q-1},\lambda_{q}}^{u_{q}}(y;y_{q})$, hence $\mathrm{weight}_{P,\lambda}(y;y)=0$. Therefore if there are any negative factors in the product the entire product vanishes, and it follows that $\mathrm{weight}_{P,\lambda}(y;y)$ is Graham-nonnegative.
\end{proof}

While Theorem \ref{theorem:main} fails to be symmetric in that the roles of $u$ and $v$ cannot be interchanged, in the equivariant case it is. Some general results include a positive formula for multiplying a double Schubert polynomial $\sch_u(x;y)$ by a factorial Schur polynomial $s_\lambda(x;y)$ where the factorial Schur polynomial has at least as many $x$ variables as the double Schubert polynomial, which is an equivariant analog of the formula of Kohnert in \cite{kohnert1997multiplication} (though not in the sense of using the same combinatorial elements). In particular this also gives a Littlewood-Richardson rule for equivariant cohomology of the Grassmannian, and for multiplying pairs of factorial Schur functions that have different numbers of variables, which is an equivariant analog of the main result of \cite{grass}. Most generally, the separated descents case in \cite{kzj} and \cite{Huang_2022} is covered in equivariant cohomology.

\begin{example}[Grassmannian case]
This computes the example in Figure 8 of \cite{knutson2003puzzles}. Let $u=[2,4,1,3]$, let $v=[1,3,2]$, and let $w=[2,4,1,3]$. Then
\begin{align*}
\monom{v}&=[3,1,2]\\
\code(\monom v)&=(2)\\
\lambda(v)&=(1,1)\\
v^{-1}\monom{v}&=[2,1]\\
wv^{-1}\monom{v}&=[4,2,1,3]
\end{align*}
Then $c_{uv}^w(y;y)=y_4-y_1$, as witnessed by the following paths.
\begin{center}
\begin{tabular}{ccc}
2&\multicolumn{1}{|c}{\mycircled{2}}&4\\
\cline{2-3}
4&4&2\\
1&1&1\\
3&3&3\\
\multicolumn{3}{c}{$y_{2} - y_{1}$}
\end{tabular}
\hspace{5pt}
\begin{tabular}{ccc}
2&\multicolumn{1}{|c}{4}&\mycircled{4}\\
\cline{2-3}
4&2&2\\
1&1&1\\
3&3&3\\
\multicolumn{3}{c}{$y_{4} - y_{2}$}
\end{tabular}
\end{center}
The formula in \cite{knutson2003puzzles} also has two terms that combine to yield the $y_4-y_1$, but the weight of the two terms is $y_4-y_3$ and $y_3-y_1$, which are different from the elements of our formula.

Now let $w=[3,4,1,2]$. Then $wv^{-1}\monom{v}=[4,3,1,2]$ and $c_{uv}^w(y;y)=1$, as witnessed by the following path:
\begin{center}
\begin{tabular}{ccc}
2&\multicolumn{1}{|c}{3}&4\\
\cline{2-3}
4&4&3\\
1&1&1\\
3&2&2\\
\multicolumn{3}{c}{1}
\end{tabular}
\end{center}
There is a third term in the expansion of the product, but in \cite{knutson2003puzzles} this is $0$ because it falls outside the cohomology ring of the chosen Grassmannian.
\end{example}

\begin{example}
We multiply a double Schubert polynomial by a factorial Schur polynomial that has more $x$ variables. Namely, we multiply the factorial Schur polynomial $s_{(2,1)}(x_1,x_2,x_3,x_4;y)$, which is equal to $\sch_{[1,2,4,6,3,5]}(x;y)$, by $\sch_{[3,1,4,2]}(x_1,x_2,x_3;y)$. Hence we set $u=[1,2,4,6,3,5]$, $v=[3,1,4,2]$. Then 
\begin{align*}
\monom{v}&=[3,4,1,2]\\
\code(\monom v)&=(2,2)\\
\lambda(v)&=(2,2)\\
v^{-1}\monom{v}&=[1,3,2]
\end{align*}
We write $w$ on top of the diagram to save space.
\begin{center}
\begin{tabular}{ccc}
\multicolumn{3}{c}{\makebox[0pt]{$[3,1,4,6,2,5]$}}\\
1&\multicolumn{1}{|c}{2}&3\\
2&\multicolumn{1}{|c}{4}&\mycircled{4}\\
\cline{2-3}
4&1&1\\
6&6&6\\
3&3&2\\
5&5&5\\
\multicolumn{3}{c}{$y_{4} - y_2$}
\end{tabular}
\hspace{12pt}
\begin{tabular}{ccc}
\multicolumn{3}{c}{\makebox[0pt]{$[3,1,5,6,2,4]$}}\\
1&\multicolumn{1}{|c}{2}&3\\
2&\multicolumn{1}{|c}{4}&5\\
\cline{2-3}
4&1&1\\
6&6&6\\
3&3&2\\
5&5&4\\
\multicolumn{3}{c}{1}
\end{tabular}
\hspace{12pt}
\begin{tabular}{ccc}
\multicolumn{3}{c}{\makebox[0pt]{$[3,1,6,4,2,5]$}}\\
1&\multicolumn{1}{|c}{2}&3\\
2&\multicolumn{1}{|c}{4}&6\\
\cline{2-3}
4&1&1\\
6&6&4\\
3&3&2\\
5&5&5\\
\multicolumn{3}{c}{1}
\end{tabular}
\hspace{12pt}
\begin{tabular}{ccc}
\multicolumn{3}{c}{\makebox[0pt]{$[3,2,4,6,1,5]$}}\\
1&\multicolumn{1}{|c}{2}&3\\
2&\multicolumn{1}{|c}{3}&4\\
\cline{2-3}
4&4&2\\
6&6&6\\
3&1&1\\
5&5&5\\
\multicolumn{3}{c}{1}
\end{tabular}
\hspace{12pt}
\begin{tabular}{ccc}
\multicolumn{3}{c}{\makebox[0pt]{$[4,1,6,2,3,5]$}}\\
1&\multicolumn{1}{|c}{2}&4\\
2&\multicolumn{1}{|c}{4}&6\\
\cline{2-3}
4&1&1\\
6&6&2\\
3&3&3\\
5&5&5\\
\multicolumn{3}{c}{1}
\end{tabular}
\end{center}
For $w=[3,1,4,6,2,5]$ there are two paths, but the weight of one of them is $y_1-z_1$ which specializes to $0$. There are several other paths for other values of $w$ that specialize to $0$ when we set $z=y$ which we do not include in the calculation.

Thus
\begin{align*}
s_{(2,1)}(x_1,x_2,x_3,x_4;y)\sch_{[3,1,4,2]}(x;y)=&(y_4-y_2)\sch_{[3,1,4,6,2,5]}(x;y)+\sch_{[3,1,5,6,2,4]}(x;y)\\
                                                  &+\sch_{[3,1,6,4,2,5]}(x;y)+\sch_{[3,2,4,6,1,5]}(x;y)\\
																									&+\sch_{[4,1,6,2,3,5]}(x;y)
\end{align*}

\begin{example}
We give an example which does not involve factorial Schur functions. Set $u=[1,3,5,4,2]$, $v=[3,1,4,2]$. Then
\begin{align*}
\monom{v}&=[3,4,1,2]\\
\code(\monom v)&=(2,2)\\
\lambda(v)&=(2,2)\\
v^{-1}\monom{v}&=[1,3,2]
\end{align*}
\begin{center}
\begin{tabular}{ccc}
\multicolumn{3}{c}{\makebox[0pt]{$[3,1,5,4,2]$}}\\
1&\multicolumn{1}{|c}{3}&\mycircled{3}\\
3&\multicolumn{1}{|c}{5}&\mycircled{5}\\
\cline{2-3}
5&1&1\\
4&4&4\\
2&2&2\\
\multicolumn{3}{c}{\makebox[0pt]{$(y_{3} - y_{2})(y_{5} - y_{2})$}}
\end{tabular}
$\phantom{(y_5-y_2)}$
\begin{tabular}{ccc}
\multicolumn{3}{c}{\makebox[0pt]{$[3,2,5,4,1]$}}\\
1&\multicolumn{1}{|c}{2}&3\\
3&\multicolumn{1}{|c}{\mycircled{3}}&5\\
\cline{2-3}
5&1&1\\
4&4&4\\
2&2&2\\
\multicolumn{3}{c}{\makebox[0pt]{$y_{3} - y_{1}$}}
\end{tabular}
$\phantom{(y_5-y_2)}$
\begin{tabular}{ccc}
\multicolumn{3}{c}{\makebox[0pt]{$[3,2,5,4,1]$}}\\
1&\multicolumn{1}{|c}{2}&3\\
3&\multicolumn{1}{|c}{5}&\mycircled{5}\\
\cline{2-3}
5&3&2\\
4&4&4\\
2&1&1\\
\multicolumn{3}{c}{\makebox[0pt]{$y_{5} - y_{2}$}}
\end{tabular}
\end{center}
$$(y_3-y_2)(y_5-y_2)\sch_{[3,1,5,4,2]}(x;y)+((y_3-y_1)+(y_5-y_2))\sch_{[3,2,5,4,1]}(x;y)$$
\begin{center}
\begin{tabular}{ccc}
\multicolumn{3}{c}{\makebox[0pt]{$[3,4,5,1,2]$}}\\
1&\multicolumn{1}{|c}{3}&\mycircled{3}\\
3&\multicolumn{1}{|c}{4}&5\\
\cline{2-3}
5&5&4\\
4&1&1\\
2&2&2\\
\multicolumn{3}{c}{\makebox[0pt]{$y_{3} - y_{2}$}}
\end{tabular}
\phantom{$(y_5-y_2)$}
\begin{tabular}{ccc}
\multicolumn{3}{c}{\makebox[0pt]{$[4,1,5,3,2]$}}\\
1&\multicolumn{1}{|c}{3}&4\\
3&\multicolumn{1}{|c}{5}&\mycircled{5}\\
\cline{2-3}
5&1&1\\
4&4&3\\
2&2&2\\
\multicolumn{3}{c}{\makebox[0pt]{$y_{5} - y_{2}$}}
\end{tabular}
\phantom{$(y_5-y_2)$}
\begin{tabular}{ccc}
\multicolumn{3}{c}{\makebox[0pt]{$[3,1,6,4,2,5]$}}\\
1&\multicolumn{1}{|c}{3}&\mycircled{3}\\
3&\multicolumn{1}{|c}{5}&6\\
\cline{2-3}
5&1&1\\
4&4&4\\
2&2&2\\
6&6&5\\
\multicolumn{3}{c}{\makebox[0pt]{$y_{3} - y_{2}$}}
\end{tabular}
\end{center}
$$(y_3-y_2)\sch_{[3,4,5,1,2]}(x;y)+(y_5-y_2)\sch_{[4,1,5,3,2]}(x;y)+(y_3-y_2)\sch_{[3,1,6,4,2,5]}(x;y)$$
\begin{center}
\begin{tabular}{ccc}
\multicolumn{3}{c}{\makebox[0pt]{$[3,4,5,2,1]$}}\\
1&\multicolumn{1}{|c}{2}&3\\
3&\multicolumn{1}{|c}{4}&5\\
\cline{2-3}
5&5&4\\
4&3&2\\
2&1&1\\
\multicolumn{3}{c}{\makebox[0pt]{$1$}}
\end{tabular}
\phantom{$(y_5-y_2)$}
\begin{tabular}{ccc}
\multicolumn{3}{c}{\makebox[0pt]{$[4,2,5,3,1]$}}\\
1&\multicolumn{1}{|c}{2}&4\\
3&\multicolumn{1}{|c}{4}&5\\
\cline{2-3}
5&5&2\\
4&3&3\\
2&1&1\\
\multicolumn{3}{c}{\makebox[0pt]{$1$}}
\end{tabular}
\phantom{$(y_5-y_2)$}
\begin{tabular}{ccc}
\multicolumn{3}{c}{\makebox[0pt]{$[4,3,5,1,2]$}}\\
1&\multicolumn{1}{|c}{3}&4\\
3&\multicolumn{1}{|c}{4}&5\\
\cline{2-3}
5&5&3\\
4&1&1\\
2&2&2\\
\multicolumn{3}{c}{\makebox[0pt]{$1$}}
\end{tabular}
\end{center}
$$\sch_{[3,4,5,2,1]}(x;y)+\sch_{[4,2,5,3,1]}(x;y)+\sch_{[4,3,5,1,2]}(x;y)$$
\begin{center}
\begin{tabular}{ccc}
\multicolumn{3}{c}{\makebox[0pt]{$[3,2,6,4,1,5]$}}\\
1&\multicolumn{1}{|c}{2}&3\\
3&\multicolumn{1}{|c}{5}&6\\
\cline{2-3}
5&3&2\\
4&4&4\\
2&1&1\\
6&6&5\\
\multicolumn{3}{c}{\makebox[0pt]{$1$}}
\end{tabular}
\phantom{$(y_5-y_2)$}
\begin{tabular}{ccc}
\multicolumn{3}{c}{\makebox[0pt]{$[4,1,6,3,2,5]$}}\\
1&\multicolumn{1}{|c}{3}&4\\
3&\multicolumn{1}{|c}{5}&6\\
\cline{2-3}
5&1&1\\
4&4&3\\
2&2&2\\
6&6&5\\
\multicolumn{3}{c}{\makebox[0pt]{$1$}}
\end{tabular}
\end{center}
$$\sch_{[3,2,6,4,1,5]}(x;y)+\sch_{[4,1,6,3,2,5]}(x;y)$$


\end{example}
\end{example}

\section{The Pieri formula} \label{section:pieriproof}

\begin{definition}
Suppose $A$ is a finite set of positive integers, and suppose the distinct elements of $A$ are $a_1,\ldots,a_m$ in increasing order. If $P(y;z)$ is a polynomial in $y$ and $z$, we write $P(y_A;z)$ to mean
$$P(y_{a_1},y_{a_2},\ldots,y_{a_m};z)$$
The case where $m$ is smaller than the actual number of $y$ variables occurring in $P(y;z)$ will not happen in the exposition, and as such we do not define the notation in this case.
\end{definition}

\begin{theorem}[The Pieri formula] \label{theorem:pieri}
Let $u\in S_\infty$ and $p\leq k$. Then
$$\sch_u(x;y)E_{p,k}(x;z)=\sum_{u\tom{k} w}E_{p-\ell(u,w),k-\ell(u,w)}(y_{P_k(u,w)};z)\sch_w(x;y)$$
\end{theorem}

\begin{example} \label{example:pieriproduct}
Let $u=[4,3,5,1,2]$, let $p=3$, and let $k=4$. Then
\begin{align*}
\sch_u(x;y)E_{p,k}(x;z)&=E_{3,4}(y_1,y_3,y_4,y_5;z)\sch_{[4,3,5,1,2]}(x;y)\\
											&+E_{2,3}(y_3,y_4,y_5;z)\sch_{[4,3,5,2,1]}(x;y)\\
                       &+E_{2,3}(y_1,y_3,y_4;z)\sch_{[4,3,6,1,2,5]}(x;y)\\
											 &+E_{1,2}(y_3,y_4;z)\sch_{[4,3,6,2,1,5]}(x;y)\\
											 &+E_{1,2}(y_1,y_4;z)\sch_{[4,5,6,1,2,3]}(x;y)\\
											 &+E_{1,2}(y_1,y_3;z)\sch_{[5,3,6,1,2,4]}(x;y)\\
											 &+\sch_{[4,5,6,2,1,3]}(x;y)+\sch_{[5,3,6,2,1,4]}(x;y)+\sch_{[5,4,6,1,2,3]}(x;y)
\end{align*}

\end{example}

The Pieri formula will be derived from Proposition \ref{proposition:minipieri}, the proof of which will require extensive setup. First, we give a formula for multiplication by $x_i-z_j$ for general $i,j$.

\begin{lemma} \label{lemma:singlevariable}
Let $u\in S_\infty$ and let $i,j\geq 1$. Then
$$(x_i-z_j)\sch_u(x;y)=(y_{u(i)}-z_j)\sch_u(x;y)+\sum_{\substack{q\neq i\\\ell(ut_{iq})=\ell(u)+1}}\mathrm{sign}(q-i)\sch_{ut_{iq}}(x;y)$$
\end{lemma}
\begin{proof}
It can be shown by the recurrence relation for $\partial_u^w$ that $\partial_u^u=u$. Thus we have
$$\partial_u^u(x_i-z_j)=x_{u(i)}-z_j$$
Substituting $y$ for $x$, we obtain that the coefficient of $\sch_u(x;y)$ is $y_{u(i)}-z_j$.

To compute the constant coefficients we use the equivariant Chevalley formula
$$\sch_u(x;y)\sch_{s_i}(x;y)=(\sch_{s_i}(u(x);y))|_{x=y}\sch_u(x;y)+\sum_{\substack{u\tom{i} u'\\\ell(u')=\ell(u)+1}}\sch_{u'}(x;y)$$
We have that
$$x_i-z_j=\sch_{s_i}(x;y)-\sch_{s_{i-1}}(x;y)+y_i-z_j$$
with the negative term $\sch_{s_{i-1}}(x;y)$ excluded if $i=1$. Thus the terms occurring with constant coefficients in the expansion of $(x_i-z_j)\sch_u(x;y)$ are
$$\sum_{\substack{u\tom{i} u'\\\ell(u,u')=1}}\sch_{u'}(x;y)-\sum_{\substack{u\tom{i-1}u''\\\ell(u,u'')=1}}\sch_{u''}(x;y)$$
The overlap in the two sums is when $u'=u''=ut_{qk}$ where $i\notin \{q,k\}$, which cancel, and what remains is the negative terms $ut_{qi}$ where $q<i$ and the positive terms $ut_{iq}$ where $q>i$. Therefore,
$$\partial_{u}^{ut_{qk}}(x_i-z_j)=0$$
if $i\notin \{q,k\}$, and if $q\neq i$ then
$$\partial_u^{ut_{iq}}(x_i-z_j)=\mathrm{sign}(q-i)$$
The result follows.
\end{proof}


\newcommand{\ntom}[1]{\centernot{\tom{#1}}}

\begin{definition}[Cycles, cycle length, disjoint cycles]
Recall that a \emph{cycle} or \emph{cyclic permutation} is a permutation $c$ written with alternative notation (not window notation) as $c=(a_1,\ldots,a_m)$ for positive integers $a_i$ such that $a_i\neq a_j$ if $i\neq j$ characterized by the fact that $c(a_i)=a_{i+1}$ for all $1\leq i\leq m$, with the index wrapping around so that $a_{m+1}=a_1$, and $c(j)=j$ if $j\notin \{a_1,\ldots,a_m\}$. Cycles $c$ have a \emph{cycle length} which is equal to the value $m$, most often distinct from their length $\ell(c)$ as a permutation. A cycle of cycle length $m$ is called an $m$-cycle. Note that $t_{ab}=(a,b)$ is a $2$-cycle. We refer to the ``elements'' of the cycle $c$ as the integers $a_i$ and say that $c$ \emph{contains} $a_i$ as though it were a set.

For an $m$-cycle $(a_1,\ldots,a_m)$ with $m>2$ we have for any $1<j<m$ that
$$(a_1,\ldots,a_m)=(a_1,\ldots,a_j)(a_j,\ldots,a_m)$$
and we also have that
$$(a_1,\ldots,a_m)=(a_m,a_1,\ldots,a_{m-1})$$
hence the cycle is unchanged by cyclic permutation of the indices in its representation in this form. Two cycles $c_1=(a_1,\ldots,a_p)$ and $c_2=(b_1,\ldots,b_q)$ are said to be \emph{disjoint} if $a_i\neq b_j$ for all $i,j$. It is a basic fact proved in an undergraduate abstract algebra course that every permutation has a unique decomposition as a product of pairwise disjoint cycles (up to commutation).
\end{definition}

We present here in Lemma \ref{lemma:piericycle} an alternative characterization of the Pieri relation $\tom{k}$ in terms of cycles that is used as the definition in \cite{robinsonpieri}. Satisfaction of these conditions is usually easier to check than Definition \ref{definition:pieri} since it is a characterization in terms of the unique disjoint cycle decomposition of $u^{-1}w$. We note, and leave the proof to the reader, that in the Pieri relation $u\tom{k} w$ if we have a sequence of reflections $t_{a_1b_1}$, $\ldots$, $t_{a_pb_p}$ realizing this then we may assume by rearranging the reflections that $b_i\leq b_j$ whenever $i\leq j$.

\begin{lemma} \label{lemma:piericycle}
Suppose $u,w\in S_\infty$ and let $k>0$ be an integer. Then $u\tom{k} w$ if and only if there exist pairwise disjoint cyclic permutations $c_1,\ldots,c_n$ of respective cycle lengths $p_1+1,\ldots,p_n+1$ such that
$$w=uc_1\cdots c_n$$
and
$$\ell(u,w)=p_1+\cdots+p_n$$
such that for each $i$ there exist pairwise distinct positive integers $a_1,\ldots,a_{p_i}$ and $b$ (different for each $i$) such that we may write the cycle $c_i$ as
$$c_i=(a_{p_i},a_{p_i-1},\ldots, a_1,b)$$
where $a_j\leq k$ for all $1\leq j\leq p_i$, $b>k$, and
$$u(a_{p_i})<\cdots<u(a_1)<u(b)$$
\end{lemma}
\begin{proof}
Suppose first that $u\tom{k} w$. Then there exist reflections $t_{a_1b_1},\ldots,t_{a_mb_m}$ for some $m$ with $a_i\neq a_j$ if $i\neq j$, $a_i\leq k<b_i$ for all $i$, $\ell(ut_{a_1b_1}\cdots t_{a_ib_i})=\ell(u)+i$ for all $i$, $b_i\leq b_j$ if $i\leq j$, and $w=ut_{a_1b_1}\cdots t_{a_mb_m}$. Fix $b$ such that there is an index $1\leq q\leq m$ such that $b=b_q$ and let $i$ and $p$ be such that $a_i,\ldots,a_{i+p}$ is the maximal portion of the sequence such that $b_j=b$ for all $i\leq j\leq i+p$. Let
$$c=t_{a_ib_i}\cdots t_{a_{i+p}b_{i+p}}$$
Given $j$ with $i\leq j\leq i+p$ we claim that if we define a permutation $c^{(j)}$ with
$$c^{(j)}=t_{a_ib}\cdots t_{a_jb}$$
then
$$c^{(j)}=(a_j,a_{j-1},\ldots,a_i,b)$$
and that
$$u(a_j)<\cdots<u(a_i)<u(b)$$
This is clear if $j=i$. Otherwise, suppose
$$c^{(j-1)}=t_{a_ib}\cdots t_{a_{j-1}b}=(a_{j-1},\ldots,a_i,b)$$
Then
$$c^{(j-1)}t_{a_jb}=(a_{j-1},\ldots,a_i,b)(b,a_j)=(a_{j-1},\ldots,a_i,b,a_j)=(a_j,\ldots,a_i,b)$$
as desired. We have that 
$$u(a_j)=uc^{(j-1)}(a_j)<uc^{(j-1)}(b)=u(a_{j-1})$$
and hence the desired claim follows by induction. Hence,
$$c=c^{(i+p)}=(a_{i+p},\ldots,a_i,b)$$
and
$$u(a_{i+p})<\cdots<u(a_i)<u(b)$$
Ranging over all $b$, these cycles, which we index in order of increasing $b$ as $c_1,\ldots,c_n$, are the ones that occur in the disjoint cycle decomposition of $u^{-1}w$, which we can see by the fact that the cycles are pairwise disjoint and $w=uc_1\ldots c_n$. Letting $p_i+1$ be the cycle length of each cycle $c_i$, we have that there are $p_1+\cdots+p_n=m$ reflections, and hence
$$\ell(uc_1\cdots c_n)=\ell(u)+p_1+\cdots+p_n$$
and we have the result.

We prove the converse now. Suppose the cycles $c_1,\ldots,c_n$ of cycle length $p_i+1$ for each $i$ exist satisfying the conditions, fixing $k$. Write $c_1$ as
$$c_1=(a_{p_1},\ldots,a_1,b)$$
with $u(a_{p_1})<\cdots<u(a_1)<u(b)$. Then
$$c_1=t_{a_1b}\cdots t_{a_{p_1}b}$$
We claim that $\ell(ut_{a_1b}\cdots t_{a_ib})>\ell(ut_{a_1b}\cdots t_{a_{i-1}b})$ for all $i$ (not yet claiming that the length increases by exactly $1$). We prove this by induction on $i$. For $i=1$, we note that $u(a_1)<u(b)$ by assumption, hence $\ell(ut_{a_1b})>\ell(u)$. Assume now that the result holds for $i-1$. Then
$$ut_{a_1b}\cdots t_{a_{i-1}b}(b)=u(a_{i-1})>u(a_i)=ut_{a_1b}\cdots t_{a_{i-1}b}(a_i)$$
since $a_j\neq a_i$ if $j\neq i$. Thus multiplying on the right by $t_{a_ib}$ increases the length by at least $1$. Applying this argument for each disjoint cycle $c_i$ to $\ell(uc_1\cdots c_i)$ and expanding the cycles into reflections, we obtain that successively multiplying on the right by each reflection increases the length by at least $1$ per reflection. Since there are $p_1+\cdots+p_n$ reflections and this is equal to $\ell(u,w)$, each reflection must increase the length by exactly $1$ and we have the result.
\end{proof}

Note that it is not sufficient in Lemma \ref{lemma:piericycle} to require only that $\ell(uc_i)=\ell(u)+p_i$ for all $i$, as it could occur that, for example in the case of two cycles, $\ell(uc_1)=\ell(u)+p_1$ and $\ell(uc_2)=\ell(u)+p_2$ but $\ell(uc_1c_2)\neq \ell(u)+p_1+p_2$. This can even happen if $p_1=p_2=1$, so that the cycles are reflections.

\begin{lemma} \label{lemma:nokm1k}
Let $u,w\in S_\infty$, let $k>0$ be an integer, and suppose $u\tom{k} w$. Then $u(i)\leq w(i)$ for all $i\leq k$, and $u(i)\geq w(i)$ for all $i>k$.
\end{lemma}
\begin{proof}
Let $t_{a_1b_1},\ldots,t_{a_pb_p}$ be a sequence of reflections realizing $u\tom{k} w$, with $a_i\leq k<b_i$ for all $i$ and $a_i\neq a_j$ if $i\neq j$, by which we mean $\ell(ut_{a_1b_1}\cdots t_{a_ib_i})=\ell(u)+i$ for all $i$ and $w=ut_{a_1b_1}\cdots t_{a_pb_p}$. We prove the statement by induction on $p$. If $p=0$, then $u(i)=w(i)$ for all $i$, so the result is clear. Otherwise, suppose the result holds for some $p\geq 0$; we show the result holds for $p+1$, considering a reflection $t_{a_{p+1}b_{p+1}}$ such that $a_{p+1}\leq k<b_{p+1}$ and $\ell(wt_{a_{p+1}b_{p+1}})=\ell(w)+1$, with $a_i\neq a_{p+1}$ for all $1\leq i\leq p$. By the induction hypothesis, $u(i)\leq w(i)$ for all $i\leq k$, and $u(i)\geq w(i)$ for all $i>k$. The only two positions that $wt_{a_{p+1}b_{p+1}}$ and $w$ differ are $a_{p+1}$ and $b_{p+1}$. Since $\ell(wt_{a_{p+1}b_{p+1}})=\ell(w)+1$, we have that $w(a_{p+1})<w(b_{p+1})$, hence 
$$u(a_{p+1})\leq w(a_{p+1})<w(b_{p+1})=wt_{a_{p+1}b_{p+1}}(a_{p+1})$$ 
Similarly, 
$$u(b_{p+1})\geq w(b_{p+1})>w(a_{p+1})=wt_{a_{p+1}b_{p+1}}(b_{p+1})$$
The result follows by induction.
\end{proof}

The following is easy to see, but we will be required to refer back to it and hence we state it as a lemma.

\begin{lemma} \label{lemma:numelems}
Suppose $u,w\in S_\infty$, $k>0$ is an integer, and suppose $u\tom{k} w$. Then
$$|P_k(u,w)|=k-\ell(u,w)$$
\end{lemma}
\begin{proof}
Let $p=\ell(u,w)$. Then there are reflections $t_{a_1b_1},\ldots,t_{a_pb_p}$ such that $a_i\leq k<b_i$ for all $i$,
$$\ell(ut_{a_1b_1}\cdots t_{a_ib_i})=\ell(u)+i$$
for all $i$, and $w=ut_{a_1b_1}\cdots t_{a_pb_p}$, with $a_i\neq a_j$ whenever $i\neq j$. By this last condition, the number of indices $i\leq k$ such that $u(i)\neq w(i)$ is exactly $p$; specifically, these indices are the values $a_j$. Thus $u(i)=w(i)$ for exactly $k-p=k-\ell(u,w)$ values of $i$ such that $i\leq k$.
\end{proof}

Parts of the statements of the below lemmas are presented in \cite{robinsonpieri} (the results are stated for multiplying $w$ by a reflection instead of $u$, which is equivalent via applying our lemmas to the interval $[w_0(n)w,w_0(n)u]$), but we will include proofs here in order to be complete. The purpose of the lemmas is to determine which terms remain and which terms cancel in an expansion of the Pieri product in the proof of Proposition \ref{proposition:minipieri}.

\begin{lemma} \label{lemma:pieriknotkn1}
Suppose $u,w\in S_\infty$ and let $k>1$ be an integer. If $u\tom{k} w$ but $u\ntom{k-1} w$, then there is a unique $q'>k$ such that $\ell(ut_{kq'})=\ell(u)+1$ and $ut_{kq'}\tom{k-1} w$. Furthermore, in that case there does not exist a $q<k$ such that $\ell(ut_{qk})=\ell(u)+1$ and $ut_{qk}\tom{k-1} w$.
\end{lemma}
\begin{proof}
Suppose $u\tom{k} w$ but $u\ntom{k-1} w$. Let $c_1,\ldots,c_n$ be the pairwise disjoint cycles in the decomposition of $u^{-1}w$, with the cycle length of $c_i$ being $p_i+1$. Some such cycle must contain $k$, as otherwise we would have $u\tom{k-1} w$ by Lemma \ref{lemma:piericycle}. Assume that $c_1=(a_{p_1},\ldots,a_1,b)$ is the cycle that contains $k$ and set $q'=b$. We claim that $ut_{kq'}\tom{k-1} w$ and $\ell(ut_{kq'})=\ell(u)+1$. Note that $u=(ut_{kq'})t_{kq'}$ and hence 
$$(ut_{kq'})t_{kq'}c_1\cdots c_n=w$$
Let $j$ be the index such that $a_j=k$. Then we have
\begin{align*}
t_{kq'}(a_{p_1},\ldots,a_{j+1},k,a_{j-1},\ldots,a_1,q')&=t_{kq'}(q',a_{p_1},\ldots,a_{j+1},k,a_{j-1},\ldots,a_1)\\
&=t_{kq'}(q',a_{p_1},\ldots,a_{j+1},k)(k,a_{j-1},\ldots,a_1)\\
&=t_{kq'}(k,q')(a_{p_1},\ldots,a_{j+1},q')(a_{j-1},\ldots,a_1,k)\\
&=(a_{p_1},\ldots,a_{j+1},q')(a_{j-1},\ldots,a_1,k)
\end{align*}
Note that $ut_{kq'}(k)=u(q')$, $ut_{kq'}(q')=u(k)$, and otherwise $ut_{kq'}(a)=u(a)$ for all $a\notin \{k,q'\}$. Thus,
$$ut_{kq'}(a_{p_1})<\cdots<ut_{kq'}(a_{j+1})$$
since the same is true for $u$, and 
$$ut_{kq'}(a_{j+1})=u(a_{j+1})<u(k)=ut_{kq'}(q')$$
hence the first cycle is valid for $\tom{k-1}$. Furthermore, if $j>1$ then
$$ut_{kq'}(a_{j-1})<\cdots<ut_{kq'}(a_{1})$$
for the same reason, and
$$ut_{kq'}(a_{1})=u(a_{1})<u(q')=ut_{kq'}(k)$$
hence the second cycle is valid for $\tom{k-1}$ as well, whereas if $j=1$ then there is no second cycle. Since the (up to) two new cycles together with $c_2,\ldots,c_n$ are all pairwise disjoint and $\ell(ut_{kq'},w)=\ell(u,w)-1$, the cycle lengths sum up correctly and we have that $ut_{kq'}\tom{k-1} w$ by Lemma \ref{lemma:piericycle}, as desired.

To show that the $q'$ we have chosen is unique, let $r>k$ be any index other than $q'$ such that $\ell(ut_{kr})=\ell(u)+1$. Suppose first that $r$ is not in any of the cycles $c_i$. Recall that $c_1$ is the cycle that contains $k$. Then $t_{kr}c_1$ is a cycle in the disjoint cycle decomposition of $(ut_{kr})^{-1}w$ containing both $k$ and $r$, which are both larger than $k-1$. This violates the condition of Lemma \ref{lemma:piericycle} that in order for $ut_{kr}\tom{k-1} w$ we must have that each cycle contains exactly one element larger than $k-1$, hence we have $ut_{kr}\ntom{k-1} w$. If $r$ is in one of the cycles, since $r>k$ and $r\neq q'$ this cycle cannot be $c_1$. Say the cycle containing $r$ is $c_2=(a_{p_2}',\ldots,a_1',r)$. Then 
\begin{align*}
t_{kr}c_1c_2&=t_{kr}(r,a_{p_2}',\ldots,a_1')(a_{p_1},\ldots,a_1,q')\\
            &=(r,a_{p_2}',\ldots,a_1',k)(k,a_{j-1},\ldots,q',a_{p_1},\ldots,a_{j+1})\\
						&=(r,a_{p_2}',\ldots,a_1',k,a_{j-1},\ldots,q',a_{p_1},\ldots,a_{j+1})
\end{align*}
is a single cycle disjoint from all the others containing both $k$ and $r$,  which are both greater than $k-1$, hence we have $ut_{kr}\ntom{k-1} w$ by Lemma \ref{lemma:piericycle}. The uniqueness of $q'$ follows.

To prove the nonexistence of $q$, consider the same cycle decomposition $c_1,\ldots,c_n$. Let $q<k$ be such that $\ell(ut_{qk})=\ell(u)+1$.  If $c_1$ (the cycle that contains $k$) does not contain $q$, then $t_{qk}c_1$ is a cycle containing both $k$ and some integer larger than $k$. If some other cycle $c_2$ contains $q$, $t_{qk}c_1c_2$ is disjoint from all the other cycles and contains both $k$ and some integer larger than $k$, similarly to the calculation above. Thus if $c_1$ does not contain $q$ then the disjoint cycle decomposition of $(ut_{qk})^{-1}w$ contains the cycle $t_{qk}c_1$ or $t_{qk}c_1c_2$, which implies that $ut_{qk}\ntom{k-1} w$ by Lemma \ref{lemma:piericycle} since the cycle contains two elements larger than $k-1$. To handle the case where $c_1$ contains $q$, in that case $t_{qk}$ splits $c_1$ into two disjoint cycles as follows. Assume $c_1=(a_{p_1},\ldots,q')$ with $a_j=k$ and $a_{j'}=q$. We must have that $j'>j$ since $u(q)<u(k)$. Then 
\begin{align*}
t_{qk}c_1&=t_{qk}(a_{p_1},\ldots,q,\ldots,k,\ldots,q')\\
         &=t_{qk}(q,\ldots,k,\ldots,q',a_{p_1},\ldots,a_{j'+1})\\
				 &=t_{qk}(q,a_{j'-1},\ldots,k)(k,\ldots,q',a_{p_1},\ldots,a_{j'+1})\\
				 &=(q,a_{j'-1},\ldots,a_{j+1})(a_{p_1},\ldots,a_{j'+1},k,\ldots,q')
\end{align*}
The cycle $(a_{p_1},\ldots,k,\ldots,q')$ contains two elements larger than $k-1$, implying that $ut_{qk}\ntom{k-1} w$ by Lemma \ref{lemma:piericycle}, as desired.
\end{proof}

\begin{example}
Let $u=[3,1,6,5,2,4]$, let $w=[4, 2, 7, 6, 1, 3, 5]$, and let $k=4$. We have that 
$$u^{-1}w=(1,6)(2,5)(4,3,7)$$
Thus $u\tom{k} w$, but $u\ntom{k-1} w$ because of the third cycle containing two elements greater than $k-1$. We set $q'=7$ according to the proof. Then
$$ut_{kq'}=[3,1,6,7,2,4,5]$$
Then we see the cycle splits by
$$t_{kq'}u^{-1}w=(1,6)(2,5)t_{kq'}(4,3,7)=(1,6)(2,5)t_{kq'}(4,7)(3,4)=(1,6)(2,5)(3,4)$$
and indeed we have that $ut_{kq'}\tom{k-1} w$.
\end{example}

\begin{lemma} \label{lemma:pieritogether}
Suppose $u,w\in S_\infty$ and let $k>1$ be an integer. If $u\tom{k} w$ and $u\tom{k-1} w$, then there is no $q<k$ such that $\ell(ut_{qk})=\ell(u)+1$ and $ut_{qk}\tom{k-1} w$ and there is no $q'>k$ such that $\ell(ut_{kq'})=\ell(u)+1$ and $ut_{kq'}\tom{k-1} w$.
\end{lemma}
\begin{proof}
Let $c_1,\ldots,c_n$ be the pairwise disjoint cycles with $c_i$ a $(p_i+1)$-cycle in the decomposition of $u^{-1}w$. These cycles, since the disjoint cycle decomposition is unique, must realize both $u\tom{k}w$ and $u\tom{k-1} w$. Thus none of these cycles can contain $k$; if one of them did, it would be impossible that $u\tom{k-1} w$ because the cycle would have two elements larger than $k-1$. 

We show first that element $q'>k$ cannot exist. Suppose $q'>k$ is such that $\ell(ut_{kq'})=\ell(u)+1$. If none of the cycles $c_i$ contains $q'$, then $(k,q')$ is a free cycle in the disjoint cycle decomposition of $(ut_{kq'})^{-1}w$. Since both elements are greater than $k-1$, we then cannot have that $ut_{kq'}\tom{k-1} w$. If one of the cycles does contain $q'$, say $c_1$, then since $c_1$ does not contain $k$ we have that $t_{kq'}c_1$ is a $(p_1+2)$-cycle containing both $k$ and $q'$; again, this implies that $ut_{kq'}\ntom{k-1} w$ because both elements are greater than $k-1$. It follows that no $q'$ exists with $\ell(ut_{kq'})=\ell(u)+1$ and $ut_{kq'}\tom{k-1} w$.

Now we consider $q$. Suppose $q<k$ is such that $\ell(ut_{qk})=\ell(u)+1$. Then $u(q)<u(k)\leq w(k)$, the second inequality by Lemma \ref{lemma:nokm1k} since $u\tom{k} w$. We then have that $ut_{qk}(k)=u(q)<w(k)$, hence by Lemma \ref{lemma:nokm1k} we cannot have that $ut_{qk}\tom{k-1} w$ since this would require that $ut_{qk}(k)\geq w(k)$. The result follows.
\end{proof}

\begin{lemma} \label{lemma:pierikn1notk}
Suppose $u,w\in S_\infty$ and let $k>1$ be an integer. If $u\tom{k-1} w$ and $u\ntom{k} w$, then there is a unique $q<k$ such that $\ell(ut_{qk})=\ell(u)+1$ and $ut_{qk}\tom{k-1} w$. When this is the case, then $P_{k-1}(ut_{qk},w)=P_{k-1}(u,w)\cup \{u(k)\}$, and for any $q'>k$ such that $\ell(ut_{kq'})=\ell(u)+1$ we have that $ut_{kq'}\ntom{k-1} w$.
\end{lemma}
\begin{proof}
Let $t_{a_1b_1},\ldots,t_{a_pb_p}$ be a sequence of reflections realizing $u\tom{k-1} w$, and assume $b_i\leq b_j$ whenever $i\leq j$. Since $u\ntom{k} w$, we must have that $b_i=k$ for some $i$. Let $i$ be minimal with this property and set $q=a_i$. Then $\ell(ut_{qk})=\ell(u)+1$. Let $c_1,\ldots,c_n$ be the pairwise disjoint cycles in the decomposition of $u^{-1}w$ with $c_i$ having cycle length $p_i+1$, and suppose the cycle with $q$ is $c_1$. $c_1$ then also must contain $b_i=k$. We then have that $t_{qk}c_1$ is the $p_1$-cycle we obtain by deleting $q$. We can see this by writing $c_1=(a_{p_1}',\ldots,q,k)$, then cyclically permuting to obtain $(q,k,a_{p_1}',\ldots,a_2')=(q,k)(a_{p_1}',\ldots,a_2',k)$, and multiplying this on the left by $t_{qk}$ cancels the $(q,k)$. We have that $ut_{qk}(k)=u(q)$, hence 
$$ut_{kq}(a_{p_1}')<\cdots<ut_{kq}(a_2')<ut_{qk}(k)$$
hence $ut_{qk}\tom{k-1} w$. 

To show uniqueness of $q$, let $r<k$ be such that $r\neq q$ and $\ell(ut_{rk})=\ell(u)+1$, and let the cycle $c_1$ be as in the previous paragraph. If $r$ is not in any of the cycles, then $t_{rk}c_1$ has cycle length $p_1+2$, hence the length condition is not satisfied in Lemma \ref{lemma:piericycle}, implying that $ut_{rk}\ntom{k-1} w$. A similar length argument applies when $r$ is in a cycle other than $c_1$, say $c_2$, with $t_{rk}c_1c_2$ being a single cycle that is too long. Suppose instead that $r$ is in $c_1$, say $a_j'=r$ with $j\geq 2$. Write
$$c_1=(k,a_{p_1}',\ldots,a_{j+1}',a_j')(a_j',a_{j-1}',\ldots,a_2',q)=(a_j',k)(k,a_{p_1}',\ldots,a_{j+1}')(a_j',a_{j-1}',\ldots,a_2',q)$$
The $(a_j',k)=(r,k)$ cancels in $t_{rk}c_1$ and we obtain two cycles, one of which is nontrivial (because it contains both $r$ and $q$) and contains only elements less than or equal to $k-1$, contradicting $ut_{rk}\tom{k-1} w$. Uniqueness of $q$ follows.

Now we prove that $P_{k-1}(ut_{qk},w)=P_{k-1}(u,w)\cup \{u(k)\}$. Note that $u(i)=ut_{qk}(i)$ if $i<k$ and $i\neq q$. Thus, if $u(i)\in P_{k-1}(u,w)$ and $i\neq q$, then $u(i)\in P_{k-1}(ut_{qk},w)$. Since $P_{k-1}(ut_{qk},w)$ must have exactly one more element than $P_{k-1}(u,w)$ by Lemma \ref{lemma:numelems}, this element must occur at index $q$. Thus $ut_{qk}(q)=w(q)=u(k)$, hence the claim follows. 


Suppose finally that $q'>k$ is such that $\ell(ut_{kq'})=\ell(u)+1$; we show that $ut_{kq'}\ntom{k-1} w$. This is due to the fact that, again assuming $c_1$ contains $k$, $t_{kq'}c_1$ contains both $k$ and $q'$. If in addition there is a cycle $c_2$ containing $q'$, then $t_{kq'}c_1c_2$ is again a single cycle containing both $k$ and $q'$. In either case, since both $k$ and $q'$ are larger than $k-1$, this contradicts the conditions in Lemma \ref{lemma:piericycle}, implying that $ut_{kq'}\ntom{k-1} w$. The result follows.
\end{proof}

\begin{example}
Let $u=[3,1,6,5,2,4]$, let $w=[5,3,6,1,2,4]$, and let $k=4$. We have that
$$u^{-1}w = (2,1,4)$$
Then $u\tom{k-1} w$ because $u(2)<u(1)<u(4)$, but $u\ntom{k} w$ because there is no element greater than $4$ in the cycle. We also have
$$ut_{1,4}t_{2,4}=w$$
Thus we choose $q=1$, since this is the first index such that the higher index is $k$. Then
$$ut_{qk}=[5,1,6,3,2,4]$$
and
$$t_{qk}(2,1,4)=(1,4)(1,4)(2,4)=(2,4)$$
Hence $ut_{qk}\tom{k-1} w$. We have that $P_{k-1}(u,w)=\{6\}$, and $P_{k-1}(ut_{qk},w)=\{5,6\}=P_{k-1}(u,w)\cup \{u(k)\}$.
\end{example}

\begin{lemma} \label{lemma:pieribothnot}
Suppose $u,w\in S_\infty$ and let $k>1$ be an integer. Suppose also that $u\ntom{k} w$ and $u\ntom{k-1} w$. Then there exists a $q'>k$ such that $\ell(ut_{kq'})=\ell(u)+1$ and $ut_{kq'}\tom{k-1} w$ if and only if there exists a $q<k$ such that $\ell(ut_{qk})=\ell(u)+1$ and $ut_{qk}\tom{k-1} w$. If either exists, then each is unique, and $P_{k-1}(ut_{qk},w)=P_{k-1}(ut_{kq'},w)$.
\end{lemma}
\begin{proof}
Let $q'>k$ be such that $\ell(ut_{kq'})=\ell(u)+1$ and $ut_{kq'}\tom{k-1} w$. Suppose also that $u\ntom{k} w$. We find a $q<k$ with $\ell(ut_{qk})=\ell(u)+1$ and $ut_{qk}\tom{k-1} w$ and afterwards show that $q'$ is unique. Let $c_1,\ldots,c_n$ be the pairwise disjoint cycles in the decomposition of $(ut_{kq'})^{-1}w$, with $c_i$ of cycle length $p_i+1$. We claim that there must be some such cycle that contains $k$. If neither $k$ nor $q'$ were present, then we would have that $u\tom{k} w$ by adding this disjoint cycle $(k,q')$ to the decomposition of $(ut_{kq'})^{-1}w$, which we assumed is not the case. The same is true if $q'$ is present but not $k$. To see this, say $c_2$ is the $(p_2+1)$-cycle containing $q'$. Then $t_{kq'}c_2$ is a ($p_2+2$)-cycle satisfying the conditions of Lemma \ref{lemma:piericycle} for $u$, $w$, and $k$ (which we can see from the fact that this is a sequence of reflections of the form $t_{aq'}$, each increasing the length by $1$ and the lower indices being pairwise distinct), and we assumed that $u\ntom{k} w$.

In the case where $k$ is present but not $q'$, say in the cycle $c_1$, suppose
$$c_1=(a_{p_1},\ldots,a_1,k)$$
Then we have
$$t_{kq'}c_1=(q',k)(k,a_{p_1},\ldots,a_1)=(q',k,a_{p_1},\ldots,a_1)=(k,a_{p_1},\ldots,a_1,q')$$
If $u(k)<u(a_{p_1})$, then $u\tom{k} w$, which we assumed is not the case. Therefore $u(k)>u(a_{p_1})$. Choose the minimal $j$ such that $u(a_j)<u(k)$, and write $t_{kq'}c_1$ as
$$(k,a_{p_1},\ldots,a_j)(a_j,\ldots,q')=(a_j,k)(k,a_{p_1},\ldots,a_{j+1})(a_j,\ldots,q')=(a_j,k)(a_{p_1},\ldots,a_{j+1},k)(a_j,\ldots,q')$$
and set $q=a_j$. Then $ut_{qk}\tom{k-1} w$ because $ut_{qk}(a_{p_1})<\cdots<ut_{qk}(k)=ut_{kq'}(a_j)$, and if $j>1$ then 
$$ut_{qk}(a_j)=u(k)<u(a_{j-1})=ut_{qk}(a_{j-1})<\cdots<ut_{qk}(q')\mathrm{,}$$
whereas if $j=1$ then $ut_{qk}(q)=u(k)<u(q')=ut_{qk}(q')$ since $\ell(ut_{kq'})=\ell(u)+1$.

Finally consider the case where both $k$ and $q'$ are present. Suppose
$$c_1=(a_{p_1},\ldots,a_1,k)$$
and
$$c_2=(a_{p_2}',\ldots,a_1',q')$$
Then
$$t_{kq'}c_1c_2=(k,a_{p_1},\ldots,a_1,q')(q',a_{p_2}',\ldots,a_1')=(k,a_{p_1},\ldots,a_1,q',a_{p_2}',\ldots,a_1')$$
Since $u\ntom{k} w$, we have that $u(a_1')>u(k)$ or $u(k)>u(a_{p_1})$ by the conditions in Lemma \ref{lemma:piericycle}, and since $u(k)=ut_{kq'}(q')>ut_{kq'}(a_1')$ given that $ut_{kq'}\tom{k-1} w$ it follows that $u(k)>u(a_{p_1})$. Choose the minimal $j$ such that $u(a_j)<u(k)$. Then we obtain
$$(k,a_{p_1},\ldots,a_{j+1},a_j)(a_j,\ldots,a_1,q',a_{p_2}',\ldots,a_1')=(a_j,k)(k,a_{p_1},\ldots,a_{j+1})(a_{p_2}',\ldots,a_1',a_j,\ldots,a_1,q')$$
and we may set $q=a_j$ as before, since then
$$ut_{qk}(a_1')=ut_{kq'}(a_1')<ut_{kq'}(q')=u(k)=ut_{qk}(a_j)$$
and if $j>1$ then
$$ut_{qk}(a_j)=u(k)<ut_{kq'}(a_{j-1})=ut_{qk}(a_{j-1})$$ 
whereas if $j=1$ then
$$ut_{qk}(a_1)=u(k)<u(q')=ut_{qk}(q')$$
since $\ell(ut_{kq'})=\ell(u)+1$, so that $ut_{qk}\tom{k-1} w$. Thus, given such a $q'$, we have identified an integer $q<k$ such that $\ell(ut_{qk})=\ell(u)+1$ and $ut_{qk}\tom{k-1} w$. To see that $q'$ is unique, we note that $q'$ is the unique element contained in the cycle containing $k$ in the disjoint cycle decomposition of $u^{-1}w$ that is larger than $k$.

We show now that if $u\ntom{k-1} w$, then if $q<k$ is such that $\ell(ut_{qk})=\ell(u)+1$ and $ut_{qk}\tom{k-1} w$ then there exists a $q'>k$ such that $\ell(ut_{kq'})=\ell(u)+1$ and $ut_{kq'}\tom{k-1} w$, and we show that $q$ is unique assuming additionally that $u\ntom{k} w$. Let $q<k$ be such that $\ell(ut_{qk})=\ell(u)+1$ and $ut_{qk}\tom{k-1} w$, and assume $u\ntom{k-1} w$. Let the pairwise disjoint cycles in the decomposition of $(ut_{qk})^{-1}w$ be $c_1,\ldots,c_n$, with $c_i$ having length $p_i+1$ (note that we have redefined the cycles here, and they are not the same as the previous paragraph). If there were no cycle that contained $q$, then we claim that we would then have that $u\tom{k-1} w$, which we assumed is not the case. To see this, suppose no cycle contains $q$ and no cycle contains $k$. Then $t_{qk}$ is a free cycle in the decomposition of $u^{-1}w$ and its existence implies $u\tom{k-1}w$ by Lemma \ref{lemma:piericycle}. If instead no cycle contained $q$ but some cycle contained $k$, then say this cycle containing $k$ is $c_2$. Then $t_{qk}c_2$ can be written as a product of reflections of the form $t_{ak}$ for pairwise distinct choices of $a$ that are less than $k$, and we can deduce that this forms a $(p_2+2)$-cycle $t_{qk}c_2$ that together with $c_1,c_3,\ldots,c_n$ imply that $u\tom{k-1} w$ by Lemma \ref{lemma:piericycle}. Thus as claimed some cycle must contain $q$.

Assume without loss of generality that $c_1$ is the cycle containing $q$. Then $c_1$ cannot contain $k$ because $ut_{qk}(q)>ut_{qk}(k)$. Thus assume 
$$c_1=(a_{p_1},\ldots,a_1,q')$$
with $q'>k$ and let $j$ be such that $a_j=q$. We have
\begin{align*}
 t_{qk}c_1&=t_{qk}(a_{j},\ldots,q')(q',a_{p_1},\ldots,a_{j+1})\\
          &=(k,a_j,\ldots,a_1,q')(q',a_{p_1},\ldots,a_{j+1})\\
					&=(k,q')(a_j,\ldots,a_1,k)(a_{p_1},\ldots,a_{j+1},q')\\
					&=t_{kq'}(a_j,\ldots,a_1,k)(a_{p_1},\ldots,a_{j+1},q')
\end{align*}
Since $ut_{qk}(a_1)<ut_{qk}(q')$, we have that $ut_{kq'}(a_1)<ut_{kq'}(k)$ if $j>1$, and in that case
$$ut_{kq'}(a_j)=u(q)<u(k)=ut_{qk}(a_j)<ut_{qk}(a_{j-1})=ut_{kq'}(a_{j-1})$$
We also have, if $j=1$,
$$ut_{kq'}(a_1)=u(q)<u(k)<u(q')=ut_{kq'}(k)$$
Thus
$$ut_{kq'}(a_j)<\cdots<ut_{kq'}(a_1)<ut_{kq'}(k)$$
so that the first cycle is valid for $ut_{kq'}\tom{k-1} w$. Since $ut_{qk}(a_{j+1})<ut_{qk}(q)=u(k)$, we have that $ut_{kq'}(a_{j+1})<ut_{kq'}(q')=u(k)$. If there is no other $c_i$ containing $k$, then we obtain that $\ell(ut_{kq'})=\ell(u)+1$ and $ut_{kq'}\tom{k-1} w$ by Lemma \ref{lemma:piericycle}. If there exists a nontrivial cycle containing $k$, say $c_2=(a_{p_2}',\ldots,a_1',k)$, then $t_{qk}c_1$ combines with $c_2$ to obtain
\begin{align*}
t_{qk}c_1c_2&=t_{kq'}(a_j,\ldots,a_1,k,a_{p_2}',\ldots,a_1')(a_{p_1},\ldots,a_{j+1},q')\\
             &=t_{kq'}(a_{p_2}',\ldots,a_1',a_j,\ldots,a_1,k)(a_{p_1},\ldots,a_{j+1},q')\mathrm{.}
\end{align*}						
Since $ut_{qk}(a_1')<ut_{qk}(k)$, we have $ut_{kq'}(a_1')<ut_{kq'}(q)=ut_{kq'}(a_j)$, hence $ut_{kq'}\tom{k-1} w$. Hence, given $q$, we have identified a $q'>k$ such that $\ell(ut_{kq'})=\ell(u)+1$ and $ut_{kq'}\tom{k-1} w$. 

Assume now that $u\ntom{k} w$; in that case, we know that the $q'$ we have identified is unique. Since 
$$u(k)=ut_{qk}(a_j)<ut_{qk}(a_{j-1})=u(a_{j-1})$$ 
we have that $q$ is the rightmost element in the cycle containing $k$, with $k$ held fixed as the rightmost element of the cycle, in the disjoint cycle decomposition of $(ut_{kq'})^{-1}w$ such that $u(q)<u(k)$, hence $q$ is unique. Thus, if $u\ntom{k} w$ and $u\ntom{k-1}w$, then $q'$ exists if and only if $q$ does, and each is unique, as desired.

For the last part we show that $P_{k-1}(ut_{qk},w)=P_{k-1}(ut_{kq'},w)$ if $q$ and $q'$ exist. We have that $ut_{qk}(i)=ut_{kq'}(i)$ if $i\neq q$ and $i<k$, and the cardinalities of $P_{k-1}(ut_{qk},w)$ and $P_{k-1}(ut_{kq'},w)$ are the same by Lemma \ref{lemma:numelems}. If $ut_{qk}(q)\in P_{k-1}(ut_{qk},w)$, then $ut_{kq'}(q)\notin P_{k-1}(ut_{kq'},w)$, and vice versa, because the values at these indices are not equal and hence they cannot both be equal to $w(q)$. Hence, neither $ut_{qk}(q)$ nor $ut_{kq'}(q)$ is equal to $w(q)$, for if one of them were then one of the two sets $P_{k-1}(ut_{qk},w)$ and $P_{k-1}(ut_{kq'},w)$ would have more elements than the other. We must therefore have that $P_{k-1}(ut_{qk},w)=P_{k-1}(ut_{kq'},w)$, as stated.
\end{proof}

\begin{example}
Let $u=[3,1,2,5,6,4]$, let $w=[5,1,6,2,3,4]$, and let $k=4$. We have that
$$u^{-1}w=(1,4,3,5)$$
Thus $u\ntom{k} w$ because $u(3)<u(4)$, and $u\ntom{k-1} w$ because both $4$ and $5$ are present in the cycle. However, if we let $q'=5$, then
$$ut_{kq'}=[3,1,2,6,5,4]$$
and
$$(ut_{kq'})^{-1}w=(1,5)(3,4)$$
and hence $ut_{kq'}\tom{k-1} w$. Since $u^{-1}w=(1,4,3,5)$, this is the case where both $k$ and $q'=5$ are present. Thus we should be able to pick $q=3$ to have $ut_{qk}\tom{k-1} w$, since this is the minimal (only) element in the cycle after $k$ such that $u(q)<u(k)$. Calculating this, we get
$$ut_{qk}=[3,1,5,2,6,4]$$
and
$$(ut_{qk})^{-1}w=(1,3,5)$$
Thus $ut_{qk}\tom{k-1} w$. One can check that this is the only index that satisfies the condition. Going in the opposite direction, since $(ut_{qk})^{-1}w=(1,3,5)$ we would have picked $q'=5$ via the proof. Additionally note that $P_{k-1}(ut_{qk},w)=\{1\}=P_{k-1}(ut_{kq'},w)$.
\end{example}

\begin{proof}[Proof of Proposition \ref{proposition:minipieri}]
We are trying to prove that 
$$\sch_u(x;y)E_k(x;z_j)=\sum_{u\tom{k} w}\mathrm{weight}_{u,k}^w(y;z_j)\sch_w(x;y)$$
We note that
$$\mathrm{weight}_{u,k}^w(y;z_j)=E_{k-\ell(u,w)}(y_{P_k(u,w)};z_j)=\prod_{i\in P_k(u,w)}(y_i-z_j)$$
so, re-expressing the desired result, we want to prove that
$$\sch_u(x;y)E_k(x;z_j)=\sum_{u\tom{k} w}\left(\prod_{i\in P_k(u,w)}(y_i-z_j)\right)\sch_w(x;y)$$
We prove this by induction on $k$. For $k=1$, the result is covered by Lemma \ref{lemma:singlevariable}, where there are no negative terms. Otherwise, assume the result holds for $k-1$. Note that
$$E_k(x;z_j)=E_{k-1}(x;z_j)(x_k-z_j)$$
We apply Lemma \ref{lemma:singlevariable} to multiply $\sch_u(x;y)$ by $x_k-z_j$. Namely, the computation of $\sch_u(x;y)E_k(x;z_j)$ becomes
$$\sch_u(x;y)(x_k-z_j)E_{k-1}(x;z_j)=(y_{u(k)}-z_j)\sch_u(x;y)E_{k-1}(x;z_j)+\sum_{\ell(ut_{kq})=\ell(u)+1}\mathrm{sign}(q-k)\sch_{ut_{kq}}(x;y)E_{k-1}(x;z_j)$$
Applying the induction hypothesis to expand the terms on the right hand side by multiplying the double Schubert polynomial by the factor $E_{k-1}(x;z_j)$, we obtain
\begin{equation} \label{equation:first}
\sum_{w:u\tom{k-1} w}(y_{u(k)}-z_j)\left(\prod_{i\in P_{k-1}(u,w)}(y_{i}-z_j)\right)\sch_w(x;y)
\end{equation}
as well as
\begin{equation} \label{equation:new}
\sum_{\substack{q'>k\\\ell(ut_{kq'})=\ell(u)+1}}\sum_{w:ut_{kq'}\tom{k-1} w}\left(\prod_{i\in P_{k-1}(ut_{kq'},w)}(y_{i}-z_j)\right)\sch_{w}(x;y)
\end{equation}
and
\begin{equation} \label{equation:neg}
-\sum_{\substack{q<k\\\ell(ut_{qk})=\ell(u)+1}}\sum_{w:ut_{qk}\tom{k-1} w}\left(\prod_{i\in P_{k-1}(ut_{qk},w)}(y_{i}-z_j)\right)\sch_{w}(x;y)
\end{equation}
The terms we want to eliminate are as follows
\begin{enumerate}
\item \label{item:badw} Terms in (\ref{equation:first}) such that $w(k)\neq u(k)$.
\item \label{item:extraw} Terms in (\ref{equation:new}) such that $u\ntom{k} w$.
\item \label{item:allneg} All terms in (\ref{equation:neg}).
\end{enumerate}
For all arguments here we assume that $\ell(ut_{ab})=\ell(u)+1$ for all reflections $t_{ab}$ we consider. For each $w$ occurring in a term of type (\ref{item:badw}), we identify a term in the sum (\ref{equation:neg}) that cancels it. Assume then that $w$ is such that $u\tom{k-1} w$ but $u(k)\neq w(k)$. Then $u(k)>w(k)$ by Lemma \ref{lemma:nokm1k}; by the same lemma, this makes it impossible that $u\tom{k} w$.  By Lemma \ref{lemma:pierikn1notk}, there exists a unique $q<k$ such that $ut_{qk}\tom{k-1} w$. Since $P_{k-1}(ut_{qk},w)=P_{k-1}(u,w)\cup \{u(k)\}$ by the same lemma, the term corresponding to $w$ in (\ref{equation:neg}) corresponding to the permutation $ut_{qk}$ cancels the term corresponding to $w$ in (\ref{equation:first}), applying the induction hypothesis. Thus for each term of type (\ref{item:badw}) we have identified a unique term in (\ref{equation:neg}) that cancels it.

Let $q'>k$ be a positive integer and let $w$ be an index such that there is a term corresponding to $ut_{kq'}$ and $w$ of type (\ref{item:extraw}) in (\ref{equation:new}), meaning that $u\ntom{k} w$. By definition we have that $ut_{kq'}\tom{k-1} w$. We cannot have that $u\tom{k-1} w$ by Lemma \ref{lemma:pierikn1notk} since the existence of the index $q'>k$ contradicts the conclusion of the lemma, so we have that $u\ntom{k} w$, $u\ntom{k-1} w$, and $q'>k$ is such that $ut_{kq'}\tom{k-1} w$. Therefore, by Lemma \ref{lemma:pieribothnot}, $q'$ is unique and there is a unique $q<k$ such that $ut_{qk}\tom{k-1}w$. By the same lemma, $P_{k-1}(ut_{qk},w)=P_{k-1}(ut_{kq'},w)$, so the two terms corresponding to these indices have the same weight with opposite signs by the induction hypothesis, hence they cancel. Thus all terms of type (\ref{item:extraw}) are canceled from the sum with a term from (\ref{equation:neg}) that was not canceled in the previous paragraph by a term of type (\ref{item:badw}) since $u\ntom{k-1} w$.

We claim that at this point all terms in (\ref{equation:neg}) have been canceled. Fix $w$. Note first that if there is a $q<k$ such that $ut_{qk}\tom{k-1} w$, then by Lemma \ref{lemma:pieritogether} we cannot have that both $u\tom{k}w$ and $u\tom{k-1}w$. Thus either $u\ntom{k-1} w$ or $u\ntom{k} w$. In the case where both are true we identified a unique corresponding $q'>k$ such that $ut_{kq'}\tom{k-1} w$ and we showed that the corresponding term has the same coefficient with the opposite sign. Thus all terms with $u\ntom{k-1} w$ and $u\ntom{k} w$ have been canceled. If $u\ntom{k-1} w$ but $u\tom{k} w$, by Lemma \ref{lemma:pieriknotkn1} $\sch_w(x;y)$ has a coefficient of $0$ in (\ref{equation:neg}). Finally, suppose $u\tom{k-1} w$ but $u\ntom{k}w$. Then the $q<k$ corresponding to the term in the sum such that $ut_{qk}\tom{k-1} w$ is the unique index with this property by Lemma \ref{lemma:pierikn1notk}, and as argued previously it has been canceled. The claim follows.

At this point the only terms that remain are those such that $u\tom{k} w$ in (\ref{equation:first}) and (\ref{equation:new}). The weight in (\ref{equation:first}) is correct by the induction hypothesis because we canceled the terms with $u(k)\neq w(k)$, and for (\ref{equation:new}), fixing $q'>k$, since $ut_{kq'}$ differs from $u$ only at $k$ and $q'$ we have that $P_k(u,w)=P_{k-1}(ut_{kq'},w)$. By Lemma \ref{lemma:pieritogether}, no $w$ that occurs in (\ref{equation:first}) also occurs in (\ref{equation:new}). By Lemma \ref{lemma:pieriknotkn1}, if $u\ntom{k-1} w$ and $u\tom{k} w$ there is exactly one $q'>k$ such that $ut_{kq'}\tom{k-1} w$, so all $\sch_w(x;y)$ such that $u\tom{k} w$ occur exactly once and we have the result.
\end{proof}

\begin{proof}[Proof of Theorem \ref{theorem:pieri}]
This follows from Proposition \ref{proposition:minipieri} coupled with the Cauchy formula for double Schubert polynomials \cite{notes}. Let $e_{i,k}(x)$ be the usual elementary symmetric polynomial $e_i(x_1,\ldots,x_k)$ and let $h_{i,k}(x)=h_i(x_1,\ldots,x_k)$ be the complete homogeneous symmetric polynomial of degree $i$ in $x_1,\ldots,x_k$. Note that when we expand the product in the definition of $E_k(x;z_1)$ we obtain
$$E_k(x;z_1)=\sum_{i=0}^k e_{i,k}(x)(-z_1)^{k-i}$$
By Proposition \ref{proposition:minipieri}, we have
\begin{align*}
\sum_{i=0}^k\sch_u(x;y)e_{i,k}(x)(-z_1)^{k-i}&=\sch_u(x;y)E_k(x;z_1)\\
&=\sum_{u\tom{k} w}E_{k-\ell(u,w)}(y_{P_k(u,w)};z)\sch_w(x;y)\\
&=\sum_{u\tom{k} w}\sum_{i=\ell(u,w)}^k e_{i-\ell(u,w),k-\ell(u,w)}(y_{P_k(u,w)})\sch_w(x;y)(-z_1)^{(k-\ell(u,w))-(i-\ell(u,w))}
\end{align*}
Equating the coefficients of powers of $-z_1$ in the above equation we obtain that
$$\sch_u(x;y)e_{i,k}(x)=\sum_{u\tom{k} w}{e_{i-\ell(u,w),k-\ell(u,w)}(y_{P_k(u,w)})\sch_w(x;y)}$$
By the Cauchy formula in \cite{notes} we have that
$$E_{p,k}(x;z)=\sum_{i=0}^p e_{i,k}(x)h_{p-i,k+1-p}(-z)$$
Therefore the coefficient of $\sch_w(x;y)$ in $\sch_u(x;y)E_{p,k}(x;z)$ is, replacing the index $i$ with $j+\ell(u,w)$,
$$\sum_{j=0}^{p-\ell(u,w)} e_{j,k-\ell(u,w)}(y_{P_k(u,w)})h_{p-\ell(u,w)-j,(k-\ell(u,w))+1-(p-\ell(u,w))}(-z)$$
which is exactly $E_{p-\ell(u,w),k-\ell(u,w)}(y_{P_k(u,w)};z)$, and we have the result.
\end{proof}

Via the following formula, our Pieri formula is also positive in the sense of Conjecture \ref{conjecture:positive} as well as the equivariant specialization.

\begin{proposition} \label{proposition:elemformula}
Let $p,k$ with $1\leq p\leq k$ be integers. Let $\mathcal{S}_{p,k}$ be the set of sequences $(a_0,\ldots,a_{p-1})$ such that $1\leq a_i\leq k+1-p$ for all $i$ and $a_i\leq a_{i+1}$ for all $0\leq i<p-1$. Then
$$E_{p,k}(x;z)=\sum_{(a_0,\ldots,a_{p-1})\in\mathcal{S}_{p,k}}\prod_{i=0}^{p-1}(x_{a_i+i}-z_{a_i})$$
\end{proposition}
\begin{proof}
This follows from the formula \cite[(4)]{molev1999littlewood}. A term in a factorial Schur function is written as the product of linear factors $x_{T(\alpha)}-z_{T(\alpha)+c-r}$, where $T$ is a tableau, $\alpha$ ranges over all boxes, and $\alpha$ is in column $c$ and row $r$. For $E_{p,k}(x;z)$, this tableau is on a single column, so that $c=1$ and $1\leq r\leq p$, and $T$  is strictly increasing along the rows.
\end{proof}

\begin{example}
We select some of the coefficients in Example \ref{example:pieriproduct} to see that they are nonnegative when substituting $z=y$.
\begin{align*}
E_{3,4}(y_1,y_3,y_4,y_5;z)&=(y_1-z_1)(y_3-z_1)(y_4-z_1)+(y_1-z_1)(y_3-z_1)(y_5-z_2)\\
                          &+(y_1-z_1)(y_4-z_2)(y_5-z_2)+(y_3-z_2)(y_4-z_2)(y_5-z_2)\\
&\Rightarrow(y_3-y_2)(y_4-y_2)(y_5-y_2)\\
E_{2,3}(y_3,y_4,y_5;z)&=(y_3-z_1)(y_4-z_1)+(y_3-z_1)(y_5-z_2)+(y_4-z_2)(y_5-z_2)\\
&\Rightarrow(y_3-y_1)(y_4-y_1)+(y_3-y_1)(y_5-y_2)+(y_4-y_2)(y_5-y_2)\\
E_{1,2}(y_1,y_3;z)&=(y_1-z_1)+(y_3-z_2)\\
&\Rightarrow y_3-y_2
\end{align*}
\end{example}

\section{A more general result} \label{section:moregeneral}

We note that in Theorem \ref{theorem:main} the value of $p$ is unimportant as long as the condition is satisfied; the computation of the coefficient $e_{uv}^w(y;z)$ has nothing to do with $p$. Indeed, Theorem \ref{theorem:main} is not the most general possible result that comes out of this method. The most general result is the following.
\begin{theorem} \label{theorem:moregeneral}
Let $u,v\in S_\infty$ and suppose $v^{-1}\monom{v} = s_{i_1}\cdots s_{i_m}$ with $\ell(us_{i_j})>\ell(u)$ for all $j$. Then
$$\sch_u(x;y)\sch_v(x;z)=\sum_{w\in S_\infty} e_{uv}^w(y;z)\sch_w(x;y)$$
\end{theorem}
The proof of this result is virtually identical and we will not include it. One obvious case where this works is where $u$ is arbitrary and $v$ is dominant, so that $v=\monom{v}$ and $v^{-1}\monom{v}$ is the identity. We illustrate a less trivial case, computing all of the coefficients that occur in the product.

\begin{example}
Let $u=[4,1,3,2]$ and let $v=[3,1,4,2]$. Then $u$ and $v$ do not satisfy the conditions of Theorem \ref{theorem:main}, but Theorem \ref{theorem:moregeneral} applies. We have
\begin{align*}
\monom{v}&=[3,4,1,2]\\
\code(\monom v)&=(2,2)\\
\lambda(v)&=(2,2)\\
v^{-1}\monom{v}&=[1,3,2]
\end{align*}
\begin{center}
\begin{tabular}{ccc}
\multicolumn{3}{c}{$[4,1,3,2]$}\\
4&\multicolumn{1}{|c}{\mycircled{4}}&\mycircled{4}\\
1&\multicolumn{1}{|c}{\mycircled{1}}&3\\
\cline{2-3}
3&3&1\\
2&2&2\\
\multicolumn{3}{c}{\makebox[0pt]{$(y_{1} - z_{1})(y_{4} - z_{1})(y_{4} - z_{2})$}}
\end{tabular}
$\phantom{(y_1-z_1)(y_4-z_2)}$
\begin{tabular}{ccc}
\multicolumn{3}{c}{$[4,1,3,2]$}\\
4&\multicolumn{1}{|c}{\mycircled{4}}&\mycircled{4}\\
1&\multicolumn{1}{|c}{3}&\mycircled{3}\\
\cline{2-3}
3&1&1\\
2&2&2\\
\multicolumn{3}{c}{\makebox[0pt]{$(y_{3} - z_{2})(y_{4} - z_{1})(y_{4} - z_{2})$}}
\end{tabular}
$\phantom{(y_1-z_1)(y_4-z_2)}$
\begin{tabular}{ccc}
\multicolumn{3}{c}{\makebox[0pt]{$[4,2,3,1]$}}\\
4&\multicolumn{1}{|c}{\mycircled{4}}&\mycircled{4}\\
1&\multicolumn{1}{|c}{2}&3\\
\cline{2-3}
3&3&2\\
2&1&1\\
\multicolumn{3}{c}{\makebox[0pt]{$(y_{4} - z_{1})(y_{4} - z_{2})$}}
\end{tabular}
\end{center}
$$((y_1-z_1)(y_4-z_1)(y_4-z_2)+(y_3-z_2)(y_4-z_1)(y_4-z_2))\sch_{[4,1,3,2]}(x;y)+(y_4-z_1)(y_4-z_2)\sch_{[4,2,3,1]}(x;y)$$
\begin{center}
\begin{tabular}{ccc}
\multicolumn{3}{c}{\makebox[0pt]{$[5,1,3,2,4]$}}\\
4&\multicolumn{1}{|c}{\mycircled{4}}&5\\
1&\multicolumn{1}{|c}{3}&\mycircled{3}\\
\cline{2-3}
3&1&1\\
2&2&2\\
5&5&4\\
\multicolumn{3}{c}{\makebox[0pt]{$(y_{3} - z_{2})(y_{4} - z_{1})$}}
\end{tabular}
$\phantom{(y_4-z_1)}$
\begin{tabular}{ccc}
\multicolumn{3}{c}{\makebox[0pt]{$[5,1,3,2,4]$}}\\
4&\multicolumn{1}{|c}{\mycircled{4}}&5\\
1&\multicolumn{1}{|c}{\mycircled{1}}&3\\
\cline{2-3}
3&3&1\\
2&2&2\\
5&5&4\\
\multicolumn{3}{c}{\makebox[0pt]{$(y_{1} - z_{1})(y_{4} - z_{1})$}}
\end{tabular}
$\phantom{(y_4-z_1)}$
\begin{tabular}{ccc}
\multicolumn{3}{c}{\makebox[0pt]{$[5,1,3,2,4]$}}\\
4&\multicolumn{1}{|c}{5}&\mycircled{5}\\
1&\multicolumn{1}{|c}{\mycircled{1}}&3\\
\cline{2-3}
3&1&1\\
2&2&2\\
5&4&4\\
\multicolumn{3}{c}{\makebox[0pt]{$(y_{1} - z_{1})(y_{5} - z_{2})$}}
\end{tabular}
$\phantom{(y_4-z_1)}$
\begin{tabular}{ccc}
\multicolumn{3}{c}{\makebox[0pt]{$[5,1,3,2,4]$}}\\
4&\multicolumn{1}{|c}{5}&\mycircled{5}\\
1&\multicolumn{1}{|c}{3}&\mycircled{3}\\
\cline{2-3}
3&1&1\\
2&2&2\\
5&4&4\\
\multicolumn{3}{c}{\makebox[0pt]{$(y_{3} - z_{2})(y_{5} - z_{2})$}}
\end{tabular}
\end{center}
$$((y_3-z_2)(y_4-z_1)+(y_1-z_1)(y_4-z_1)+(y_1-z_1)(y_5-z_1)+(y_3-z_2)(y_5-z_2))\sch_{[5,1,3,2,4]}(x;y)$$
\begin{center}
\begin{tabular}{ccc}
\multicolumn{3}{c}{\makebox[0pt]{$[4,1,5,2,3]$}}\\
4&\multicolumn{1}{|c}{\mycircled{4}}&\mycircled{4}\\
1&\multicolumn{1}{|c}{3}&5\\
\cline{2-3}
3&1&1\\
2&2&2\\
5&5&3\\
\multicolumn{3}{c}{\makebox[0pt]{$(y_{4} - z_{1})(y_{4} - z_{2})$}}
\end{tabular}
$\phantom{(y_4-z_1)}$
\begin{tabular}{ccc}
\multicolumn{3}{c}{\makebox[0pt]{$[5,1,4,2,3]$}}\\
4&\multicolumn{1}{|c}{5}&\mycircled{5}\\
1&\multicolumn{1}{|c}{3}&4\\
\cline{2-3}
3&1&1\\
2&2&2\\
5&4&3\\
\multicolumn{3}{c}{\makebox[0pt]{$y_5- z_{2}$}}
\end{tabular}
$\phantom{(y_4-z_1)}$
\begin{tabular}{ccc}
\multicolumn{3}{c}{\makebox[0pt]{$[5,1,4,2,3]$}}\\
4&\multicolumn{1}{|c}{\mycircled{4}}&5\\
1&\multicolumn{1}{|c}{3}&4\\
\cline{2-3}
3&1&1\\
2&2&2\\
5&5&3\\
\multicolumn{3}{c}{\makebox[0pt]{$y_4-z_1$}}
\end{tabular}
\end{center}
$$(y_4-z_1)(y_4-z_2)\sch_{[4,1,5,2,3]}(x;y)+((y_5-z_2)+(y_4-z_1))\sch_{[5,1,4,2,3]}(x;y)$$
\begin{center}
\begin{tabular}{ccc}
\multicolumn{3}{c}{\makebox[0pt]{$[5,2,3,1,4]$}}\\
4&\multicolumn{1}{|c}{5}&\mycircled{5}\\
1&\multicolumn{1}{|c}{2}&3\\
\cline{2-3}
3&3&2\\
2&1&1\\
5&4&4\\
\multicolumn{3}{c}{\makebox[0pt]{$y_5-z_2$}}
\end{tabular}
$\phantom{(y_4-z_1)}$
\begin{tabular}{ccc}
\multicolumn{3}{c}{\makebox[0pt]{$[5,2,3,1,4]$}}\\
4&\multicolumn{1}{|c}{\mycircled{4}}&5\\
1&\multicolumn{1}{|c}{2}&3\\
\cline{2-3}
3&3&2\\
2&1&1\\
5&5&4\\
\multicolumn{3}{c}{\makebox[0pt]{$y_4-z_1$}}
\end{tabular}
\end{center}
$$((y_5-z_2)+(y_4-z_1))\sch_{[5,2,3,1,4]}(x;y)$$
\begin{center}
\begin{tabular}{ccc}
\multicolumn{3}{c}{\makebox[0pt]{$[6,1,3,2,4,5]$}}\\
4&\multicolumn{1}{|c}{5}&6\\
1&\multicolumn{1}{|c}{\mycircled{1}}&3\\
\cline{2-3}
3&3&1\\
2&2&2\\
5&4&4\\
6&6&5\\
\multicolumn{3}{c}{\makebox[0pt]{$y_1-z_1$}}
\end{tabular}
$\phantom{(y_4-z_1)}$
\begin{tabular}{ccc}
\multicolumn{3}{c}{\makebox[0pt]{$[6,1,3,2,4,5]$}}\\
4&\multicolumn{1}{|c}{5}&6\\
1&\multicolumn{1}{|c}{3}&\mycircled{3}\\
\cline{2-3}
3&1&1\\
2&2&2\\
5&4&4\\
6&6&5\\
\multicolumn{3}{c}{\makebox[0pt]{$y_3-z_2$}}
\end{tabular}
$\phantom{(y_4-z_1)}$
\begin{tabular}{ccc}
\multicolumn{3}{c}{\makebox[0pt]{$[6,1,4,2,3,5]$}}\\
4&\multicolumn{1}{|c}{5}&6\\
1&\multicolumn{1}{|c}{3}&4\\
\cline{2-3}
3&1&1\\
2&2&2\\
5&4&3\\
6&6&5\\
\multicolumn{3}{c}{$1$}
\end{tabular}
$\phantom{(y_4-z_1)}$
\begin{tabular}{ccc}
\multicolumn{3}{c}{\makebox[0pt]{$[6,2,3,1,4,5]$}}\\
4&\multicolumn{1}{|c}{5}&6\\
1&\multicolumn{1}{|c}{2}&3\\
\cline{2-3}
3&3&2\\
2&1&1\\
5&4&4\\
6&6&5\\
\multicolumn{3}{c}{$1$}
\end{tabular}
\end{center}
$$((y_1-z_1)+(y_3-z_2))\sch_{[6,1,3,2,4,5]}(x;y)+\sch_{[6,1,4,2,3,5]}(x;y)+\sch_{[6,2,3,1,4,5]}(x;y)$$
\end{example}

\section{Python package for computing Schubert products}

We have written a python package \verb|schubmult| that is available in the central python package repository which requires python $3.9$ or higher that allows one to compute products of ordinary Schubert polynomials, products of double Schubert polynomials in the same set of variables, and products of double Schubert polynomials in different sets of variables. When one installs the package with 
\begin{center}
\begin{BVerbatim}
pip install schubmult 
\end{BVerbatim}
\end{center}
this installs three executables on the user's system: \verb|schubmult_py| for ordinary Schubert polynomials (so named so as not to conflict with existing software named \verb|schubmult|), \verb|schubmult_double| for double Schubert polynomials in the same set of variables, with the result expressed with nonnegative coefficients in terms of the negative roots, and \verb|schubmult_yz| for computing Molev-Sagan products of double Schubert polynomials. It is entirely platform independent, provided the python installation can be configured properly on the system. We have tested it on both Linux and Windows.

It is recommended you periodically check for new versions, which you can do directly by running
\begin{center}
\begin{BVerbatim}
pip install schubmult --upgrade
\end{BVerbatim}
\end{center}
when you already have \verb|schubmult|.

The method of calculation is to express one of the (double) Schubert polynomials in the product in terms of (factorial) elementary symmetric polynomials and apply the Pieri formula. There's no choice of which one to express in this way in the Molev-Sagan case, but for \verb|schubmult_py| and \verb|schubmult_double| it picks the one that is closest to being dominant. Even in the ordinary Schubert polynomial case it is faster than currently existing software in most cases, sometimes dramatically.

All three executables have the same command line syntax, which is the same as the command line syntax for \verb|schubmult| from \verb|lrcalc|. An instructive example is to compute
\begin{center}
\begin{BVerbatim}
schubmult_py 1 2 4 9 11 6 8 12 3 5 7 10 - 6 8 1 2 3 4 7 10 12 14 5 9 11 13
\end{BVerbatim}
\end{center}
The number of permutations on the command line is not limited to two, so we could for example compute the product of three Schubert polynomials as
\begin{center}
\begin{BVerbatim}
schubmult_py 6 1 4 3 2 5 - 7 6 3 5 1 2 4 - 5 1 4 3 2
\end{BVerbatim}
\end{center}
For an example Molev-Sagan product, one could execute
\begin{center}
\begin{BVerbatim}
schubmult_yz 1 3 4 6 2 5 - 2 1 5 7 3 4 6 
\end{BVerbatim}
\end{center}
and the same product in equivariant cohomology can be computed as
\begin{center}
\begin{BVerbatim}
schubmult_double 1 3 4 6 2 5 - 2 1 5 7 3 4 6 
\end{BVerbatim}
\end{center}
where some terms will vanish. 

For all three executables, one can specify the input as the code of the permutation instead of the permutation itself with the command line argument \verb|-code|. For example,
\begin{center}
\begin{BVerbatim}
schubmult_yz -code 0 1 1 2 - 1 0 2 3
\end{BVerbatim}
\end{center}
The output will then express the permutations as codes as well.

The coefficients in \verb|schubmult_yz| are expressed in terms of $y_i-z_j$. If desired, the coefficients can be displayed positively as polynomials in $y_i-z_j$ using the command line option \verb|--display-positive|, for example
\begin{center}
\begin{BVerbatim}
schubmult_yz 1 3 4 6 2 5 - 2 1 5 7 3 4 6 --display-positive
\end{BVerbatim}
\end{center}
Feasibility of displaying the result positively varies, and one should expect to run into some cases such that the computation will not finish in a reasonable amount of time. As this is still only always possible conjecturally, it is conceivable one could run into a counterexample; if one is found, please email the counterexample to the author. 

The coefficients in \verb|schubmult_double| are subjected to a change of variables in order for them to visibly have nonnegative coefficients.

Though the package was primarily intended to be used through the installed executables, it can also be imported into your own python script and the function \verb|schubmult| used programmatically. For example,
\begin{Verbatim}[commandchars=\\\{\}]
\textcolor{blue}{from} schubmult.schubmult_yz \textcolor{blue}{import} schubmult  
  
coeff_dict = schubmult(\{(\textcolor{red}{1}, \textcolor{red}{3}, \textcolor{red}{4}, \textcolor{red}{6}, \textcolor{red}{2}, \textcolor{red}{5}): \textcolor{red}{1}\}, (\textcolor{red}{2}, \textcolor{red}{1}, \textcolor{red}{5}, \textcolor{red}{7}, \textcolor{red}{3}, \textcolor{red}{4}, \textcolor{red}{6}))
\end{Verbatim}
The return value is a python \verb|dict| where the keys are permutations, which are expressed as \verb|tuple|s of \verb|int|s, and the values are either \verb|int|s (for \verb|schubmult.schubmult_py|) or \verb|symengine| objects (which can be polynomials). The \verb|dict| representation is intended to represent a linear combination of (double) Schubert polynomials with the values as the coefficients. Note that there may be values of $0$ in the resulting \verb|dict|, or values that simplify to $0$ but do not appear immediately to be $0$ if they are \verb|symengine| objects. No $0$ values will be displayed in \verb|schubmult_py| or \verb|schubmult_double|, but there may be coefficients displayed in \verb|schubmult_yz| that simplify to $0$. Simplifying a polynomial is an expensive calculation, and in the interest of saving time this is not done in the output, except in \verb|schubmult_double|.

In addition to python, \verb|schubmult| is also functional in Sage.

Visit the pypi page at 
\begin{center}
\mbox{\textcolor{blue}{https://pypi.org/project/schubmult}} 
\end{center}
for updates on the latest version.
\bibliographystyle{acm}
\bibliography{formschub}

\begin{thebibliography}{10}

\bibitem{bbrc}
{\sc Bergeron, N., and Billey, S.}
\newblock {R}{C}-graphs and {S}chubert polynomials.
\newblock {\em Experimental Math. 2}, 4 (1993), 257--269.

\bibitem{bsskew}
{\sc Bergeron, N., and Sottile, F.}
\newblock Skew {S}chubert functions and the {P}ieri formula for flag manifolds.
\newblock {\em Trans. Amer. Math. Soc. 354}, 2 (2001), 651--673.

\bibitem{bgg}
{\sc Bernstein, I.~N., Gelfand, I.~M., and Gelfand, S.~I.}
\newblock Schubert cells and the cohomology of the spaces {G}{/}{P}.
\newblock {\em Russian Math. Surveys 28\/} (1973), 1--26.

\bibitem{combcox}
{\sc Bj{\"o}rner, A., and Brenti, F.}
\newblock {\em Combinatorics of {C}oxeter {G}roups}.
\newblock Springer, 2005.

\bibitem{buch2015mutations}
{\sc Buch, A.~S.}
\newblock Mutations of puzzles and equivariant cohomology of two-step flag
  varieties.
\newblock {\em Ann. of Math.\/} (2015), 173--220.

\bibitem{bump2011factorial}
{\sc Bump, D., McNamara, P.~J., and Nakasuji, M.}
\newblock Factorial {S}chur functions and the {Y}ang-{B}axter equation.
\newblock {\em arXiv preprint arXiv:1108.3087\/} (2011).

\bibitem{chen2004skew}
{\sc Chen, W.~Y., Yan, G.-G., and Yang, A.~L.}
\newblock The skew {S}chubert polynomials.
\newblock {\em European J. Combin. 25}, 8 (2004), 1181--1196.

\bibitem{groth_puzzle}
{\sc Fan, N.~J., Guo, P.~L., and Xiong, R.}
\newblock Bumpless pipe dreams meet {P}uzzles.
\newblock {\em arXiv preprint arXiv:2309.00467\/} (2023).

\bibitem{graham2001positivity}
{\sc Graham, W.}
\newblock Positivity in equivariant {S}chubert calculus.
\newblock {\em Duke Math. J. 109}, 3 (2001), 599--614.

\bibitem{question}
{\sc (https://mathoverflow.net/users/62135/matt samuel), M.~S.}
\newblock Product of a {S}chubert polynomial and a double {S}chubert
  polynomial.
\newblock MathOverflow.
\newblock URL:https://mathoverflow.net/q/212762 (version: 2016-09-06).

\bibitem{Huang_2022}
{\sc Huang, D.}
\newblock Schubert products for permutations with separated descents.
\newblock {\em International Mathematics Research Notices\/} (November 2022).

\bibitem{kirillov2007skew}
{\sc Kirillov, A.~N.}
\newblock Skew divided difference operators and {S}chubert polynomials.
\newblock {\em SIGMA. Symmetry, Integrability and Geometry: Methods and
  Applications 3\/} (2007), 072.

\bibitem{knutsonpipe}
{\sc Knutson, A.}
\newblock Schubert polynomials, pipe dreams, equivariant classes, and a
  co-transition formula.
\newblock {\em arXiv.org:1909.13777\/} (2019).

\bibitem{knutson2003puzzles}
{\sc Knutson, A., and Tao, T.}
\newblock Puzzles and (equivariant) cohomology of {G}rassmannians.
\newblock {\em Duke Math. J. 119}, 2 (2003), 221--260.

\bibitem{kzj}
{\sc Knutson, A., and Zinn-Justin, P.}
\newblock Schubert calculus and quiver varieties.
\newblock https://pi.math.cornell.edu/allenk/boston2019.pdf, 2019.
\newblock Accessed 2023-06-10.

\bibitem{knutson_separated}
{\sc Knutson, A., and Zinn-Justin, P.}
\newblock Schubert puzzles and integrability {I}{I}{I}: separated descents.
\newblock {\em arXiv preprint arXiv:2306.13855\/} (2023).

\bibitem{kohnert1997multiplication}
{\sc Kohnert, A.}
\newblock Multiplication of a {S}chubert polynomial by a {S}chur polynomial.
\newblock {\em Ann. Comb. 1}, 1 (1997), 367--375.

\bibitem{kostant1986nil}
{\sc Kostant, B., and Kumar, S.}
\newblock The nil {H}ecke ring and cohomology of {G}/{P} for a {K}ac-{M}oody
  group {G}.
\newblock {\em Proceedings of the National Academy of Sciences 83}, 6 (1986),
  1543--1545.

\bibitem{lsschub}
{\sc Lascoux, A., and {Sch\"utzenberger}, M.}
\newblock Polyn{\^o}mes de {S}chubert.
\newblock {\em Comptes Rendus de l{'}Acad{\'e}mie des Sciences, S{\'e}rie I
  294}, 13 (1982), 447--450.

\bibitem{lenart2003skew}
{\sc Lenart, C., and Sottile, F.}
\newblock Skew {S}chubert polynomials.
\newblock {\em Proc. Amer. Math. Soc. 131}, 11 (2003), 3319--3328.

\bibitem{notes}
{\sc Macdonald, I.~G.}
\newblock {\em Notes on Schubert Polynomials}.
\newblock Universit{\'e} Du Qu{\'e}bec {\`a} Montr{\'e}al, D{\'e}p. de
  math{\'e}matiques et d{'}informatique, 1991.

\bibitem{molev2009littlewood}
{\sc Molev, A.}
\newblock Littlewood--{R}ichardson polynomials.
\newblock {\em Journal of Algebra 321}, 11 (2009), 3450--3468.

\bibitem{molev1999littlewood}
{\sc Molev, A., and Sagan, B.}
\newblock A {L}ittlewood-{R}ichardson rule for factorial {S}chur functions.
\newblock {\em Trans. Amer. Math. Soc. 351}, 11 (1999), 4429--4443.

\bibitem{grass}
{\sc Purbhoo, K., and Sottile, F.}
\newblock A {L}ittlewood-{R}ichardson rule for {G}rassmannian permutations.
\newblock {\em Proc. Am. Math. Soc. 137\/} (2009), 1875--1882.

\bibitem{robinsonpieri}
{\sc Robinson, S.}
\newblock A {P}ieri-{T}ype {F}ormula for {$H^*T (SL_n (C)/B)$}.
\newblock {\em J. Algebra 249}, 1 (2002), 38--58.

\bibitem{samuelleibniz}
{\sc Samuel, M.~J.}
\newblock {\em The {L}eibniz formula for divided difference operators
  associated to {K}ac-{M}oody root systems}.
\newblock PhD thesis, Rutgers, The State University of New Jersey, 2014.

\bibitem{sottile}
{\sc Sottile, F.}
\newblock Pieri{'}s rule for flag manifolds and {S}chubert polynomials.
\newblock {\em Ann. Inst. Fourier 46}, 1 (1996), 89--110.

\bibitem{tamvakis2020tableau}
{\sc Tamvakis, H.}
\newblock Tableau formulas for skew {S}chubert polynomials.
\newblock {\em arXiv preprint arXiv:2008.07034\/} (2020).

\end{thebibliography}
\end{document}